%% file: mot_algorithms.tex
\documentclass[a4paper,11pt,DIV=11,
 abstract=on
 ]{scrartcl}
 
\input{preamble}

\title{Approximative Algorithms for Multi-Marginal Optimal Transport and Free-Support Wasserstein Barycenters}
\author{Johannes von Lindheim\thanks{Institute of Mathematics,
			Technische Universit\"at Berlin,
			Strasse des 17. Juni 136, 10587 Berlin, Germany,
			vonlindheim@tu-berlin.de}}
\date{\today}

\begin{document}

\maketitle

\begin{abstract}
	Computationally solving multi-marginal optimal transport (MOT) with squared Euclidean costs for $N$ discrete probability measures has recently attracted considerable attention, in part because of the correspondence of its solutions with Wasserstein-$2$ barycenters, which have many applications in data science.
	In general, this problem is NP-hard, 
	calling for practical approximative algorithms.
	While entropic regularization has been successfully applied to approximate Wasserstein barycenters, this loses the sparsity of the optimal solution, making it difficult to solve the MOT problem directly in practice because of the curse of dimensionality.
	Thus, for obtaining barycenters, one usually resorts to fixed-support restrictions to a grid, which is, however, prohibitive in higher ambient dimensions $d$.
	In this paper, after analyzing the relationship between MOT and barycenters, we present two algorithms to approximate the solution of MOT directly, requiring mainly just $N-1$ standard two-marginal OT computations.
	Thus, they are fast, memory-efficient and easy to implement and can be used with any sparse OT solver as a black box.
	Moreover, they produce sparse solutions and show promising numerical results.
	We analyze these algorithms theoretically, proving upper and lower bounds for the relative approximation error.
\end{abstract}

\section{Introduction}
The multi-marginal optimal transport problem (MOT) is an increasingly popular generalization of the classical Monge--Kantorovich optimal transport problem to several marginal measures.
It was originally introduced by \cite{GS98} in the continuous setting for squared Euclidean costs and further generalized in various ways, e.g. to entropy regularized \cite{flows19BCN,tree21HRCK} and unbalanced variants with non-exact marginal constraints \cite{UMOT21BLNS}.
For a survey with general cost functions and their applications, e.g. in economics \cite{CE10teams} or density functional theory in physics \cite{DFT12B,DFT13C}, we refer to \cite{mot15P}.

Closely related to MOT, and among its most prominent applications, are computations of barycenters in the Wasserstein space, see \cite{ABM16brute,AC11barycenters,PRV20oncomputation}.
Wasserstein barycenters have nice mathematical properties, since they are the Fr\'echet means with respect to the Wasserstein distance \cite{frechetpersistence14,frechettemplates05T,frechet19procrustes}.
Their applications range from mixing textures \cite{texturemix11,HLPR21gotex}, stippling patterns, BRDF \cite{BPPH11networksimplex} or color distributions and shapes \cite{convolutional15SPCetal} over averaging of sensor data \cite{sensors20H} to Bayesian statistics \cite{bayes18}, just to name a few.
We also refer to the surveys \cite{PC19book,stataspects19PZ}.

Unfortunately,
MOT and Wasserstein barycenters are in general hard to compute \cite{AB21nphard}.
Although there are polynomial-time methods for fixed dimension $d$ \cite{AB21fixedd}, 
there is still a need for fast approximations.
Many algorithms restrict the support of the solution to a fixed set and minimize only over the weights.
Such methods include projected subgradient \cite{CD14fast}, 
iterative Bregman projections \cite{BCC15IBP}, (proximal) algorithms 
based on the latter \cite{complexitybarycenters19KTDDGU}, 
an interior point method \cite{GWXY19MAAIPM}, Gauss-Seidel based alternating direction of multipliers \cite{LJDK21linear}, 
the multi-marginal Sinkhorn algorithm and its accelerated variant \cite{lin2019complexity}, 
the debiased Sinkhorn barycenter algorithm \cite{JCG20debiased}, 
methods using the Wasserstein distance on a tree \cite{TSKRY21treesliced}, accelerated Bregman projections \cite{LHXCJ20fastIBP} and a method based on mirror prox and one based on the dual extrapolation scheme \cite{DT21complexitybounds}, among others.

On the other hand, barycenters without such restriction are called free-support barycenters.
As we will see, they can be obtained directly from the solution of the MOT problem~\eqref{eq:mot}, which can be obtained by solving a linear program (LP) \cite{ABM16brute} that scales exponentially in $N$, however.
An exact polynomial-time method for fixed $d$ is given in \cite{AB21fixedd}, several LP based methods in \cite{improvedLP20BP,columngen22BP}, whereas approximative algorithms include another LP-based method \cite{B20LPapprox}, an inexact proximal alternating minimization method \cite{iPAM21QP} and the iterative swapping algorithm \cite{PRV20oncomputation}.
A free-support barycenter method based on the Frank--Wolfe algorithm is given in \cite{LSPC2019frankwolfe}.
Another method computes continuous barycenters using another way of parameterizing them \cite{cont20LGYS}.
Further speedups can be obtained by subsampling the given measures \cite{HMZ20randomized} or dimensionality reduction of the support point clouds \cite{ISS21dimreduction}.

Despite the plethora of literature, many algorithms with low theoretical computational complexity 
or high accuracy solutions are lacking practical applicability, 
since their implementation is rather involved or they are slow in real-world scenarios.
While iterative Bregman projections are a standard benchmark that are hard to beat in simplicity, 
fixed-support methods applied on a grid suffer from the curse of dimensionality in $d$.
Using the sparsity of the solution (see Theorem~\ref{thm:basic}), free-support methods or algorithms approximating MOT directly can overcome this, but a simple yet effective algorithm is still missing.
Thus, we present two approximative algorithms that are straightforward to implement, while our numerical experiments validate their precision and speed.
Moreover, both algorithms enjoy theoretical approximation guarantees, 
have a computational complexity independent of $d$ and produce sparse solutions, such that they are memory-efficient.
In a nutshell, they ``glue together'' an approximate solution of MOT
from $N-1$ standard two-marginal transport plans that can be found using any off-the-shelve sparse optimal transport solver as a black box.

The remainder of this paper is organized as follows.
In Section~\ref{sec:bary-mot}, we elaborate on the relation between the MOT \eqref{eq:mot} and Wasserstein barycenter problem \eqref{eq:bary}, as this will be the foundation for our theoretical analyses.
In Section~\ref{sec:algorithms}, we introduce our two algorithms and show upper and lower relative error bounds for each of them.
Further, we prove that they provide exact solutions in the case $d=1$.
In Section~\ref{sec:numerics}, we validate the accuracy and speed of the algorithms by numerical examples.
Concluding remarks are given in Section~\ref{sec:discussion}.

\section{Relation of Barycenter and MOT Problem} \label{sec:bary-mot}
Let $\mathcal P(\R^d)$ denote the space of probability measures on $\mathbb R^d$.
For some measurable function $T\colon \R^{d_1}\to \R^{d_2}$, the \emph{push-forward measure of $\pi\in \mathcal P (\R^{d_1})$} is defined as $T_\# \pi = \pi\circ T^{-1} \in \mathcal P(\R^{d_2})$.
For two discrete measures 
$\mu = \sum_{j=1}^n \mu_j \delta (x_j)$, 
$\nu = \sum_{k=1}^m \nu_k \delta (y_k)$, 
the \emph{Wasserstein-$2$-distance} 
is defined by
\[
\Wtwo^2(\mu, \nu) 
= 
\min_{\pi\in\Pi(\mu, \nu)} \langle c, \pi \rangle
=
\min_{\pi\in\Pi(\mu, \nu)} \sum_{j=1}^n \sum_{k=1}^m 
\pi_{j,k} \Vert x_{j}-y_{k}\Vert^2_2,
\]
where $c(x,y) \coloneqq \|x-y\|_2^2$ and $\Pi(\mu, \nu)$ denotes the set of probability measures
$$\pi = \sum_{j=1}^n \sum_{k=1}^m \pi_{j,k} \delta(x_j,y_k)$$
with marginals $\mu$ and $\nu$.
The above optimization problem is convex, but can have multiple minimizers $\pi$.
It can be shown that there exists a minimizer supported only on $n+m-1$ points, see \cite[Prop.~3.4]{PC19book} and Theorem \ref{thm:basic} below.
Moreover, there are several algorithms for computing a minimizer that meet this support requirement, 
e.g., the network simplex method \cite{AMO93networkflows}.

In this paper, we consider multi-marginal optimal transport problems for $N$ discrete probability measures $\mu^i\in \mathcal P(\R^d)$ supported at $\supp(\mu^i) = \{ x^i_1, \dots, x^i_{n_i}\}$, $i=1,\ldots,N$, i.e., 
\begin{equation}\label{eq:mu_def}
\mu^i= \sum_{j=1}^{n_i} \mu^i_j \delta(x^i_{j}),\quad i=1, \dots, N.
\end{equation}
To this end, we introduce for pairwise different $i_1, \ldots, i_m \in \{1,\dots, N\}$ the projections
$$
P_{i_1,\ldots,i_m} ((x_1, \dots, x_N) )= (x_{i_1},\ldots,x_{i_m})
$$
and define the set of transport plans
\[
\Pi(\mu^1, \dots, \mu^N)
=  \{ \pi\in \mathcal P((\R^d)^N): (P_i)_{\#} \pi = \mu^i, \, i=1,\dots, N \}.
\]
Given 
$\lambda = (\lambda_1,\ldots,\lambda_N) \in \Delta_{N}$, 
where
$\Delta_{N} \coloneqq \{\lambda \in (0,1)^N: \sum_{j=1}^N \lambda_j =1\}$
denotes the open probability simplex, we want to solve the MOT problem 
\begin{equation}\label{eq:mot}
\min_{\pi \in \Pi(\mu^1, \dots , \mu^N)}\Phi(\pi), 
\qquad 
\Phi(\pi) \coloneqq \langle c_{\mathrm{MOT}}, \pi\rangle,
\qquad
c_{\mathrm{MOT}}(x_1, \dots, x_N) \coloneqq \sum_{s<t}^N \lambda_{s}\lambda_{t}\Vert x_{s} - x_{t}\Vert_2^2.
\end{equation}
Every $\pi \in \Pi(\mu^1, \dots , \mu^N)$ can be written as
$$
\pi = \sum_{j_1=1}^{n_1}\dots \sum_{j_N=1}^{n_N} \pi_{j_1,\ldots,j_N} \delta (x^1_{j_1},\ldots,x^N_{j_N}),
$$
where 
$\pi_{j_1,\ldots,j_N} \in [0,1]$ and $\sum_{j_1, \ldots,j_N} \pi_{j_1,\ldots,j_N} = 1$.
Instead of this notation we will often use a representation that counts only positive summands
of pairwise different point tuples, i.e.,
if we have $M$ of such summands we write
\begin{equation} \label{eq:pi}
\pi = \sum_{j=1}^{M} \pi_{j} \delta (x_{1,j},\ldots,x_{N,j}).
\end{equation}
Here is an example for $N=2$:
\begin{align}
\pi &= \pi_{1,1} \delta(x^1_1, x^2_1) + \pi_{1,2} \delta(x^1_1, x^2_2) + \pi_{2,2} \delta(x^1_2, x^2_2) \\
&= \pi_1 \delta(x_{1,1}, x_{2,1}) + \pi_2 \delta(x_{1,2}, x_{2,2}) + \pi_3 \delta(x_{1,3}, x_{2,3}).
\end{align}
With this notation, we have for the marginals that
$$
(P_{i_1,\ldots,i_m})_{\#} \pi =  \sum_{j=1}^{M} \pi_{j} \delta (x_{i_1,j}, \ldots, x_{i_m,j}).
$$
We will need the following relation between  the MOT problem \eqref{eq:mot} and pairwise Wasserstein distances
\begin{align}
\Phi(\pi) 
&= \sum_{j=1}^M \pi_j c_{\mathrm{MOT}}(x_{1,j}, \ldots,x_{N,j})
= \sum_{j=1}^M \pi_j \sum_{s<t}^N \lambda_{s}\lambda_{t}\Vert x_{s,j} - x_{t,j}\Vert_2^2 \nonumber \\
&\ge \sum_{s<t}^N  \lambda_{s}\lambda_{t} \mathcal W_2^2(\mu^s,\mu^t).\label{eq:2-wass}
\end{align}

We are in particular interested in the relation between the MOT problem with marginals $\mu^i$, $i=1,\ldots,N$
and the Wasserstein-$2$ barycenter problem for the same measures:
\begin{equation}\label{eq:bary}
\min_{\nu \in \mathcal P(\R^d)} \Psi(\nu) , \qquad 
\Psi(\nu) \coloneqq \sumitoNlambdai \Wtwo^2(\mu^i, \nu).
\end{equation}
The relation is given via the \emph{weighted mean operator} $M_\lambda\colon \R^{N, d}\to \R^d$ defined by
\[
M_\lambda(x_1, \dots, x_N) \coloneqq \sumitoNlambdai x_i.
\]
Then we have
\begin{equation} \label{eq:mean}
(M_\lambda)_{\#} \pi = \sum_{j=1}^{M} \pi_{j} \delta (m_j), \qquad 
m_j \coloneqq M_\lambda((x_{1,j},\ldots,x_{N,j}) ) = \sumitoNlambdai x_{i,j}. 
\end{equation}
Note that $\sum_{j=1}^M \pi_j \delta(x_{i,j}, m_j) \in \Pi(\mu^i,(M_\lambda)_{\#} \pi )$.

The following lemma shows a relation between the functions $\Phi$ in the MOT problem and $\Psi$ from the barycenter problem. The proof of the lemma contains some fundamental relations that will be often used in the following.

\begin{lemma}\label{lem:psi_leq_phi}
For the functions $\Psi$ and $\Phi$ given by \eqref{eq:mot} and \eqref{eq:bary}, respectively,
we have
\begin{equation} \label{eq:psi_leq_phi}
\Phi( \pi) \ge  \Psi((M_\lambda)_{\#} \pi)
\end{equation}
for any $\pi$ in \eqref{eq:pi}.
\end{lemma}
 
\begin{proof}
By incorporating the weighted means \eqref{eq:mean}, the function $\Phi$ can be rewritten as 
\begin{align}
\Phi(\pi)
&= \frac12 \sum_{j=1}^M \pi_j \sum_{s,t=1}^N \lambda_{s}\lambda_{t}\Vert x_{s,j} - x_{t,j}\Vert_2^2\nonumber\\
&= \frac12 \sum_{j=1}^M \pi_j \Big(\sum_{s,t=1}^N \lambda_{s}\lambda_{t}\Vert x_{s,j}\Vert_2^2 
+ \sum_{s,t=1}^N \lambda_{s}\lambda_{t}\Vert x_{t,j}\Vert_2^2 
- 2 \sum_{s,t=1}^N \lambda_{s}\lambda_{t} \langle x_{s,j}, x_{t,j} \rangle\Big)\nonumber\\
&= \sum_{j=1}^M \pi_j \Big( \sum_{s=1}^N \lambda_{s}\Vert x_{s,j}\Vert_2^2 - \|m_j\|^2 \Big)\nonumber\\
&= \sum_{j=1}^M \pi_j \sum_{s=1}^N \lambda_{s} \Vert x_{s,j}- m_j\Vert_2^2. \label{eq:Phi_2}
\end{align}
Then we see that
\begin{align*}
\Phi( \pi)
=
\sumitoNlambdai \sum_{j=1}^M  \pi_j\ \Vert  x_{i, j} - m_j \Vert^2
\ge
\sumitoNlambdai \Wtwo^2(\mu^i, (M_\lambda)_{\#} \pi)
= \Psi((M_\lambda)_{\#} \pi).
\end{align*}
\end{proof}

Next, we aim to show that for an optimal solution $\hat\pi$ of the MOT problem, we do indeed have $\Phi(\hat\pi)=\Psi(\hat\nu)$.
We start by stating the following equality, which we will use frequently from now on.

\begin{lemma}\label{lem:jensen_eq}
	For any points $x_1, \dots, x_N, y\in \R^d$ and $\lambda\in \Delta_N$, we have
	\begin{equation} \label{eq:helper}
	\sum_{i=1}^N \lambda_i \Vert x_i - y\Vert^2_2 = \Vert m-y\Vert_2^2 + \sum_{i=1}^N \lambda_i \Vert x_i - m\Vert^2_2,
	\end{equation}
	where $m \coloneqq \sum_{i=1}^N \lambda_i x_i$.
\end{lemma}
\begin{proof}
	Setting $z\coloneqq m-y$ we obtain
	\begin{align*}
	\sum_{i=1}^N \lambda_i \Vert x_i - y\Vert_2^2
	&= \sum_{i=1}^N \lambda_i \Vert x_i -m+z\Vert_2^2
	= \sum_{i=1}^N \lambda_i \left(\Vert z\Vert_2^2 + \Vert x_i-m\Vert_2^2 - 2\langle x_i-m, z\rangle \right) \\
	&=  \Vert m-y\Vert_2^2 + \sum_{i=1}^N \lambda_i \Vert x_i-m\Vert_2^2.
	\end{align*}
\end{proof}

Next, we restate some results from \cite{ABM16brute} in our notation.

\begin{theorem}\label{thm:basic}
	If $\hat\pi = \sum_{j=1}^{M} \hat \pi_{j} \delta (\hat x_{1,j},\ldots,\hat x_{N,j})$ 
	is an optimal plan of $\Phi$ in \eqref{eq:mot}, then an optimal solution $\hat\nu$ of the barycenter problem \eqref{eq:bary} is
	\begin{equation}\label{eq:mccann}
	\hat\nu = (M_\lambda)_{\#} \hat\pi = \sum_{j=1}^{M} \hat \pi_{j} \delta(\hat m_j), \qquad
	\hat m_j \coloneqq \sumitoNlambdai \hat x_{i,j}.
	\end{equation}
	Conversely, each optimal barycenter $\hat \nu$ in \eqref{eq:bary} can be obtained 
	in such a way from an optimal MOT plan $\hat \pi$ in \eqref{eq:mot}.
	
	Further, there exists an optimal plan $\hat \pi$
	such that
	\begin{equation}\label{eq:sparsity}
	\#\supp(\hat\pi) \leq \sumitoN n_i - N + 1.
	\end{equation}
\end{theorem}
\begin{proof}
	Note that $\pi$ in our notation is the law of the random variable $(X_1, \dots, X_N)$ in \cite[Prop.~1]{ABM16brute}.
	Thus, setting
	\begin{align*}
		\mathbb P((X_1, \dots, X_N) = (x_{1, j}, \dots, x_{N, j})) = \pi_j,
		\quad \bar X = \frac 1 N(X_1+ \dots+ X_N),
		\quad m_j = \sumitoNlambdai x_{i, j},
	\end{align*}
	we get by rearranging \eqref{eq:helper} that
	\begin{align*}
		\mathbb E \Vert \bar X\Vert^2
		= \sum_{j=1}^M \pi_j \Vert m_j - 0\Vert^2 
		= \sum_{j=1}^M \pi_j \sum_i\frac 1 N\Big(- \Vert x_{i, j}- m_j \Vert^2 +  \Vert x_{i, j} -0 \Vert^2\Big) 
		= -\Phi(\pi) + \text{const}.
	\end{align*}
	Thus, maximizing $\mathbb E \Vert \bar X\Vert^2$ is equivalent to minimizing $\Phi(\pi)$.
	We conclude by noticing that the proofs in \cite[Prop.~1, Thm.~2]{ABM16brute} are straightforward to generalize to arbitrary $\lambda\in \Delta_N$.
\end{proof}

In particular, the theorem says that for an optimal barycenter $\hat\nu$,
\[
\supp(\hat\nu) \subseteq \Big\{ \sumitoNlambdai x^i : x^i\in \supp(\mu^i), \, i=1, \dots, N \Big\},
\]
i.e., optimal barycenters are again discrete measures supported in the convex hull of the supports of the $\mu^i$, $i=1,\ldots,N$.

Although $\Phi(\hat\pi)=\Psi(\hat\nu)$ could be derived from the proofs in \cite{ABM16brute} as well, we show this differently alongside some other relations that are, to the best of our knowledge, not yet established and characterize the relation between $\hat\pi$ and $\hat\nu$ more explicitly.

The next proposition shows that the weighted means in \eqref{eq:mean} are pairwise distinct for optimal plans of the MOT problem.

\begin{proposition}\label{prop:means_unique}
	If $\hat\pi = \sum_{j=1}^{M} \hat \pi_{j} \delta (\hat x_{1,j},\ldots,\hat x_{N,j})$ 
	is an optimal plan in \eqref{eq:mot}, then 
	$\hat m_j \not = \hat m_k$ for $j \not = k$, where
	$\hat m_j = M_\lambda (\hat x_{1,j},\ldots,\hat x_{N,j})$.
\end{proposition}
\begin{proof}
	Assume that in contrary
	$\hat m_{j} = \hat m_{k} = m$ for some $j\neq k$.
	Without loss of generality let $\hat x_{N, j} \neq \hat x_{N, k}$.
	Then we define the tuples 
	\begin{align*}
	(x_{1, j}', \dots, x_{N-1, j}', x_{N, j}') &= (\hat x_{1, j}, \dots, \hat x_{N-1, j}, \hat x_{N, k}) \\
	(x_{1, k}', \dots, x_{N-1, k}', x_{N, k}') &= (\hat x_{1, k}, \dots, \hat x_{N-1, k}, \hat x_{N, j})
	\end{align*}
	and consider for $h \coloneqq\min(\hat \pi_{j}, \hat \pi_{k})>0$ the plan
	\begin{align*}
	\pi' \coloneqq \hat\pi + h \left(
	\delta( x_{1, j}', \dots,  x_{N, j}')
	+ \delta( x_{1, k}', \dots,  x_{N, k}') 
	- \delta(\hat x_{1, j}, \dots, \hat x_{N, j})
	- \delta(\hat x_{1, k}, \dots, \hat x_{N, k}) \right)	.
	\end{align*}
	By construction, we verify that $\pi'\in \Pi(\mu^1, \dots, \mu^N)$.
	Further, we conclude as in \eqref{eq:Phi_2} that 
	\begin{align*}
	\frac 1 h (\Phi(\hat\pi) -\Phi(\pi')) 
	&= \sum_{s<t}^N \lambda_{s}\lambda_{t} \left(
	\Vert \hat x_{s, j} - \hat x_{t, j} \Vert^2
	+ \Vert \hat x_{s, k} - \hat x_{t, k} \Vert^2
	- \Vert  x_{s, j}' -  x_{t, j}' \Vert^2
	- \Vert  x_{s, k}' -  x_{t, k}' \Vert^2
	\right) \\
	&= \sumitoNlambdai \left(
	\Vert \hat x_{i, j} - m\Vert^2
	+ \Vert \hat x_{i, k} - m\Vert^2
	- \Vert x_{i, j}' - m_{j}'\Vert^2
	- \Vert x_{i, k}' - m_{k}'\Vert^2
	\right),
	\end{align*}
	where 
	$m_{l}' \coloneqq M_\lambda ( x_{1, l}', \ldots, x_{N, l}')$, $l=j, k$. 
	Finally, we obtain by definition of the tuples and Lemma \ref{lem:jensen_eq} that
	\begin{align*}
	\frac 1 h (\Phi(\hat\pi) -\Phi(\pi')) 
	&= \sumitoNlambdai \left(
	\Vert  x_{i, j}' - m\Vert^2
	+ \Vert  x_{i, k}' - m\Vert^2
	- \Vert x_{i, j}' - m_{j}'\Vert^2
	- \Vert x_{i, k}' - m_{k}'\Vert^2
	\right) \\
	&= 
	\Vert m - m_{j}'\Vert^2 + \Vert m - m_{k}'\Vert^2
	= 2\lambda_N^2\Vert \hat x_{N, j} - \hat x_{N, k}\Vert^2 > 0,
	\end{align*}
	which contradicts the optimality of $\hat\pi$.	
\end{proof}

Next, we show how the optimal transport from $\hat\nu = \push \hat \pi$ 
to $\mu^i$ is determined by $\hat\pi$.

\begin{proposition}\label{prop:ot_bary_to_mui}
	Let $\hat \pi$ be an optimal plan in \eqref{eq:mot} and $\hat\nu = \push\hat\pi$.
	Then it holds for all $i=1,\ldots,N$ that
	\begin{equation} \label{eq:non-splitting}
	\sum_{j =1}^M \hat \pi_j \delta(\hat m_j, \hat x_{i, j})
	\in \argmin_{\pi\in \Pi(\hat\nu, \mu^i)} \langle c, \pi \rangle.
	\end{equation}
\end{proposition}

\begin{proof}
	Without loss of generality we consider the case $i=N$. Assume that \eqref{eq:non-splitting} is not true.
	Since $\hat \nu$ is supported on $\hat m_j$, $j=1,\ldots,M$,
	we have that 
	$
	\pi^N \in \argmin_{\pi \in \Pi(\hat\nu, \mu^N)} \langle c, \pi\rangle 
	$ 
	can be written by counting only positive summands in the form
	$$
	\pi^N 	 = \sum_{j=1}^M \sum_{l_j} \pi_{j,l_j}^N \delta(\hat m_j, x_{N,j,l_j})
	$$
	where $x_{N,j,l_j} \in \mathrm{supp} (\mu^N)$ 
	and
	$\sum_{l_j} \pi_{j,l_j}^N = \hat \pi_j$. Note that by Proposition \ref{prop:means_unique}
	the $\hat m_j$ are pairwise different.
	Then, by assumption,
	\begin{equation} \label{eq:mot_subopt_assump}
	\sum_{j =1}^M \hat \pi_j \Vert \hat m_j -\hat x_{N, j}\Vert^2
	> \sum_{j=1}^M \sum_{l_j} \pi_{j,l_j}^N \Vert \hat m_j - x_{N,j,l_j}\Vert^2.
	\end{equation}
	Now we consider
	\[
	\pi' \coloneqq \sum_{j=1}^M \sum_{l_j} \pi_{j,l_j}^N \delta(\hat x_{1, j}, \dots, \hat x_{N-1, j}, x_{N, j, l_j})
	\]
	which is clearly in $\Pi(\mu^1,\ldots,\mu^N)$.
	Setting for $i=1,\ldots,N-1$ and all $l_j$,
	$$
	x_{i,j,l_j} \coloneqq \hat x_{i,j}  \quad \mathrm{and} \quad m_{j,l_j} = \sum_{i=1}^N \lambda_i x_{i,j,l_j}
	$$
	we see as in \eqref{eq:Phi_2} that
	\begin{align*}
	\Phi(\pi') &= \sum_{j=1}^M \sum_{l_j} \pi_{j,l_j}^N \sum_{i=1}^N \lambda_i \|x_{i,j,l_j} - m_{j,l_j} \|^2\\
	&\le 
	\sum_{j=1}^M \sum_{l_j} \pi_{j,l_j}^N 
	\Big(\|\hat m_j - m_{j,l_j}\|^2
	+ 
	\sum_{i=1}^N \lambda_i \|x_{i,j,l_j} - m_{j,l_j} \|^2
	\Big)
	\end{align*}
	and further by Lemma \ref{lem:jensen_eq} that
	\begin{align*}
	\Phi(\pi') 
	&\le 
	\sum_{j=1}^M  \sum_{l_j} \pi_{j,l_j}^N 
	\sum_{i=1}^N \lambda_i \|x_{i,j,l_j} - \hat m_{j} \|^2
	\\
	&
	=
	\sum_{j=1}^M \hat \pi_j \sum_{i=1}^{N-1} \lambda_i \|\hat x_{i,j} - \hat m_{j} \|^2
	+ 
	\sum_{j=1}^M  \sum_{l_j} \pi_{j,l_j}^N\lambda_N \|x_{N,j,l_j} - \hat m_{j} \|^2.
	\end{align*}
	Now \eqref{eq:mot_subopt_assump} implies
	\begin{align*}
	\Phi(\pi') 
	&< 
	\sum_{j=1}^M \hat \pi_j \sumitoNlambdai \Vert \hat x_{i, j} - \hat m_j\Vert^2
	= 
	\Phi(\hat \pi) 
	\end{align*}
	which contradicts the optimality of $\hat\pi$.
\end{proof}

Finally, the previous considerations lead to the desired result.
\begin{proposition}
	Let $\hat\pi$ be an optimal solution of the MOT problem in \eqref{eq:mot} and $\hat\nu$ an optimal barycenter, i.e., a minimizer in \eqref{eq:bary}. Then it holds
	\[
	\Phi(\hat\pi) = \Psi(\hat\nu).
	\]
\end{proposition}
\begin{proof}
	As a direct consequence of \eqref{eq:Phi_2} and Proposition~\ref{prop:ot_bary_to_mui}, we get that $\Phi(\hat\pi) = \Psi(\push \pi)$.
	Since $\push \pi$ is an optimal solution of \eqref{eq:bary} by Theorem~\ref{thm:basic}, it holds $\Psi(\push \pi)=\Psi(\hat\nu)$.
\end{proof}

Notably, the Propositions~\ref{prop:ot_bary_to_mui} and \ref{prop:means_unique} also show that the optimal transport from $\hat\nu$ to any $\mu^i$ is non-mass-splitting.

Next, we show a relation between the cost $\Psi(\tilde\nu)$ of any barycenter $\tilde\nu$ and the cost $\Psi(\hat\nu)$ of the optimal barycenter $\hat\nu$ and their Wasserstein distance $\Wtwo^2(\hat\nu, \tilde\nu)$.
\begin{proposition}\label{prop:cost_dist_estimate} For any discrete $\tilde\nu\in \mathcal P(\R^d)$, it holds that
	\begin{equation}\label{eq:cost_dist_estimate}
	\Psi(\tilde\nu)
	\leq \Psi(\hat\nu) + \Wtwo^2(\tilde\nu, \hat\nu).
	\end{equation}
\end{proposition}
\begin{proof}
	Since $\hat\nu = \push \hat\pi$, as above we can write
	\[
	\hat\pi= \sum_{k} \hat\pi_k\delta(\hat x_{1, k}, \dots, \hat x_{N, k}) \quad \text{and}\quad
	\hat\nu = \sum_{k} \hat\pi_k \delta(\hat m_k) \quad \text{with} \quad
	\hat m_k = \sumitoNlambdai \hat x_{i, k}.
	\]
	We also write $\tilde\nu= \sum_{j}\tilde\nu_j \tilde m_j$.
	Since $(P_i)_\#\hat\pi = \mu^i$ for all $i=1,\dots, N$, $\mu^i$ can be written as $\mu^i = \sum_k \hat\pi_k \delta(\hat x_{i, k})$.
	Let 
	\[
	\hat\pi^i
	\coloneqq \sum_k \hat\pi_k\delta(\hat m_k, \hat x_{i, k})
	\]
	be the coupling between $\hat\nu$ and $\mu^i$ given by $\hat\pi$ for all $i=1, \dots, N$.
	Next, take some $\bar\pi \in \argmin_{\pi\in \Pi(\tilde\nu, \hat\nu)} \langle c, \pi\rangle$, which can be written as
	\[
	\bar\pi = \sum_{j, k} \bar\pi_{j, k}\delta(\tilde m_j, \hat m_k)
	\]
	such that $\sum_j  \bar\pi_{j, k}=\hat\pi_k$.
	Then we define a -- not necessarily optimal -- transport plan $\tilde\pi^i\in\Pi(\tilde\nu, \mu^i)$ for every $i=1, \dots, N$ as follows, which will produce the right hand side of \eqref{eq:cost_dist_estimate}:
	\[
	\tilde\pi^i \coloneqq \sum_{j, k} \bar\pi_{j, k} \delta(\tilde m_j, \hat x_{i, k}).
	\]
	It is easy to see that $\tilde\pi^i\in \Pi(\tilde\nu, \mu^i)$.
	Thus,
	\begin{align*}
	\Psi(\tilde\nu)
	= \sumitoNlambdai \Wtwo^2(\tilde\nu, \mu^i)
	\leq \sumitoNlambdai \langle c, \tilde\pi^i \rangle 
	= \sum_{j, k} \bar\pi_{j, k} \sumitoNlambdai  \Vert \tilde m_j - \hat x_{i, k}\Vert^2
	\end{align*}
	and hence, by Lemma~\ref{lem:jensen_eq} and \eqref{eq:Phi_2},
	\begin{align*}
	\Psi(\tilde\nu)
	&\leq \sum_{j, k} \bar\pi_{j, k} \sumitoNlambdai (\Vert \tilde m_j - \hat m_k\Vert^2 + \Vert \hat m_k - \hat x_{i, k}\Vert^2) \\
	&= \sum_{j, k} \bar\pi_{j, k} \Vert \tilde m_j - \hat m_k\Vert^2 + \sum_{k} \hat\pi_k \sumitoNlambdai \Vert \hat m_k - \hat x_{i, k}\Vert^2
	= \Wtwo^2(\tilde\nu, \hat\nu) + \Phi(\hat\pi).
	\end{align*}
	By $\Phi(\hat\pi) = \Psi(\hat\nu)$, the statement follows.
\end{proof}

We conclude this section with an example, where
\begin{equation*}
\Phi(\tilde\pi) > \Phi(\hat\pi) + \Wtwo^2(\tilde\nu, \hat\nu),
\end{equation*}
i.e., Proposition~\ref{prop:cost_dist_estimate} does not hold for $\Phi$ instead of $\Psi$.
\begin{example} Set
	\begin{align*}
	x_{1, 1} = 0, \quad
	x_{1, 2} = 3, \quad
	x_{2, 1} = 1, \quad
	x_{2, 2} = 2, \quad
	x_{3, 1} = 1, \quad
	x_{3, 2} = 2.
	\end{align*}
	Set $\lambda \equiv \frac 1 3$ and $\mu^i= \frac 1 2 (\delta(x_{i, 1}) + \delta(x_{i, 2}))$ for $i=1, 2, 3$.
	Further,
	\begin{align*}
	\hat\pi &= \frac 1 2 (\delta(x_{1, 1},  x_{2, 1},  x_{3, 1}) + \delta(x_{1, 2},  x_{2, 2},  x_{3, 2})), \\
	\tilde\pi &= \frac 1 2 (\delta(x_{1, 1},  x_{2, 2},  x_{3, 2}) + \delta(x_{1, 2},  x_{2, 1},  x_{3, 1})).
	\end{align*}
	Then we have
	\begin{align*}
	\hat\nu = \frac 1 2 \Big(\delta\Big(\frac 2 3\Big) + \delta\Big(\frac 7 3 \Big)\Big), \quad
	\tilde\nu = \frac 1 2 \Big(\delta\Big(\frac 4 3\Big) + \delta\Big(\frac 5 3 \Big)\Big).
	\end{align*}
	Further,
	\begin{align*}
	\Phi(\hat\pi) = \frac 2 9 \cdot 1^2 = \frac 2 9, \quad 
	\Phi(\tilde\pi) = \frac 2 9 \cdot 2^2 = \frac 8 9, \quad 
	\Wtwo^2(\tilde\nu, \hat\nu) = \Big(\frac 2 3\Big)^2 = \frac 4 9.
	\end{align*}
	Hence
	\[
	\Phi(\tilde\pi) = \frac 8 9 > \frac 6 9 = \frac 2 9 + \frac 4 9 = \Phi(\hat\pi) + \Wtwo^2(\tilde\nu, \hat\nu).
	\]
	Nevertheless, it holds
	\[
	\Psi(\tilde\nu) = \frac 1 3 \Big( \Big( \frac 4 3 \Big)^2 + 2\cdot \Big( \frac 1 3 \Big)^2 \Big) = \frac {18} {27} = \frac 6 9 = \Phi(\hat\pi) + \Wtwo^2(\tilde\nu, \hat\nu) = \Psi(\hat\nu) + \Wtwo^2(\tilde\nu, \hat\nu)
	\]
	in alignment with Proposition~\ref{prop:cost_dist_estimate}.
	This example highlights that the optimal two-marginal transport plans from $\tilde\nu$ to the $\mu^i$ can not be directly read off of the support of $\tilde\pi$, which is in contrast to the optimal plan $\hat\pi$, see Proposition~\ref{prop:ot_bary_to_mui}.
	
	We note in passing that the same example also shows that different multi-marginal plans can have the same barycenter:
	Define
	\begin{align*}
	\pi^1 &= \frac 1 2 (\delta(x_{1, 1}, x_{2, 1}, x_{3, 2}) + \delta(x_{1, 2}, x_{2, 2}, x_{3, 1})), \\
	\pi^2 &= \frac 1 2 (\delta(x_{1, 1},  x_{2, 2},  x_{3, 1}) + \delta(x_{1, 2},  x_{2, 1},  x_{3, 2})).
	\end{align*}
	Then $(M_\lambda)_\# \pi^1 = (M_\lambda)_\# \pi^2 = \frac 1 2 (\delta(1) + \delta(2))$.
	
\end{example}

\section{Algorithms for MOT Approximation}\label{sec:algorithms}
In this section, we propose two algorithms for computing approximate MOT plans.
We will see that both algorithms require mainly the computation of $N-1$ two-marginal Wasserstein plans.

In the following, let $N$ discrete measures $\mu^i\in \mathcal P(\R^d)$ of the form \eqref{eq:mu_def} 
and $\lambda \in \Delta_N$ be given.
Then, starting with $\tilde\pi^{(1)} \coloneqq \mu^1$, the algorithms compute for $r=2,\ldots,N$ iteratively
\begin{equation}\label{eq:ref_alg}
{\tilde{\pi}}^{(r)} \in \argmin_{\pi \in \Pi(\tilde\pi^{(r-1)}, \mu^r)} \langle c_r, \pi \rangle,
\end{equation}
where 
\begin{equation}\label{eq:iteration_set}
\Pi(\tilde\pi^{(r-1)}, \mu^r) \coloneqq \{ \pi\in \mathcal P((\R^d)^r) : (P_{1, \dots, r-1})_\#\pi = \tilde\pi^{(r-1)}, (P_r)_\# \pi = \mu^r \}
\end{equation}
and the cost functions $c_r$ are given by 
\begin{align}
\mathrm{Algorithm \; 1}: \qquad  &c_r(x_1,\ldots,x_r) \coloneqq  \Vert x_1 - x_r\Vert^2, \label{eq:cost_alg_I}\\
\mathrm{Algorithm \; 2}: \qquad &c_r(x_1,\ldots,x_r) \coloneqq\big \Vert  \sum_{i=1}^{r-1} \bar\lambda_{i, r-1}  x_i - x_r \big \Vert^2, 
\label{eq:cost_alg_II}
\end{align}
where 
$$
\bar \lambda_r = (\bar \lambda_{r,1},\ldots \bar \lambda_{r,r}) \in \Delta_r, 
\qquad 
\bar \lambda_{i, r} \coloneqq \frac{\lambda_i}{\sum_{j=1}^{r} \lambda_j}.
$$
Then, for $\tilde \pi \coloneqq \tilde{\pi}^{(N)}$, we can approximate the optimal barycenter by 
$\tilde \nu \coloneqq (M_\lambda)_\#  \tilde \pi$.
Clearly, we have by construction for both cost functions that
\begin{align}
(P_i) _\#\tilde \pi &= \mu^i, \qquad i=1,\ldots,N,\\
(P_{1,i}) _\#\tilde \pi &= (P_{1,i})_\# \tilde \pi^{(r)} ,\qquad r=2,\ldots,N, \; i=1,\ldots,r.
\end{align}

Since the cost function in the first algorithm always refers to $\mu_1$ we call it \emph{reference algorithm}.
It is somewhat inspired by the recent literature on linear optimal transport, see e.g. \cite{LOT13WSBOR,LOT20MDC,LOT21MC,GWLOT21BBS,HK-LOT22CCST}.

On the other hand, the second algorithm will be called \emph{greedy algorithm}, which can be motivated as follows.
Denote by 
$$
c_{\mathrm{MOT}}^{(r)}(x_1,...,x_r)\coloneqq \sum_{i<l}\lambda_i\lambda_l\|x_i-x_l\|^2
$$
the cost function of the MOT problem \eqref{eq:mot} reduced to the first $r$ measures $\mu^1,...,\mu^r$ and set
\[
m \coloneqq \sum_{i=1}^{r-1} \bar\lambda_{i, r-1} x_i.
\]
By construction, we have $c_{\mathrm{MOT}}^{(N)}=c_{\mathrm{MOT}}$ where $c_{\mathrm{MOT}}$ is defined as in \eqref{eq:mot}.
A greedy approach would be to set the cost function $c_r$ in iteration \eqref{eq:ref_alg} to $c_r=\smash{c_{\mathrm{MOT}}^{(r)}}$.
By Lemma~\ref{lem:jensen_eq},
\begin{align*}
	c_{\mathrm{MOT}}^{(r)}(x_1,...,x_r)
	&= \sum_{i<l<r}\lambda_i\lambda_l\|x_i-x_l\|^2+\lambda_r \sum_{i=1}^{r-1}\lambda_i \|x_i-x_r\|^2 \\
	&= \sum_{i<l<r}\lambda_i\lambda_l\|x_i-x_l\|^2
	+ \lambda_r(1-\lambda_r)\sum_{i=1}^{r-1} \bar\lambda_{i, r-1} \Vert x_i - m\Vert^2
	+ \lambda_r(1-\lambda_r)  \Vert m-x_r\Vert^2.
\end{align*}
Since the first $r-1$ marginals in \eqref{eq:ref_alg} are fixed, the first two sums are just a constant, while last sum is equal to the cost function $c_r$ in \eqref{eq:cost_alg_II} up to a multiplicative constant.
Thus, the greedy approach using the cost function $\smash{c_{\mathrm{MOT}}^{(r)}}$ in \eqref{eq:ref_alg} is equivalent to using the cost function \eqref{eq:cost_alg_II}.

The optimal plans ${\tilde{\pi}}^{(r)}$ in \eqref{eq:ref_alg} are in general not unique.
We are interested in plans with small supports. In the next two subsections we have a closer look at the
computations of such plans and address the question how well these algorithms approximate the exact barycenters. 
We will use that by \eqref{eq:psi_leq_phi} and Theorem \ref{thm:basic} the relation
\[
\frac{\Psi(\tilde\nu)}{\Psi(\hat\nu)} 
= \frac{\Psi(\tilde\nu)}{\Phi(\hat\pi)}
\leq \frac{\Phi(\tilde\pi)}{\Phi(\hat\pi)}
\]
holds true, so that it is sufficient to bound only the last quotient.

Note that some theoretical upper bounds can be obtained already for certain trivial choices.
For example, for $k=\argmax_i \lambda_i$, simply taking $\tilde\nu \coloneqq \mu^k$ yields
\[
\Psi(\tilde\nu)
= \Psi(\mu^k)
= \sumitoNlambdai \Wtwo^2(\mu^k, \mu^i)
\]
and by \eqref{eq:psi_leq_phi} and \eqref{eq:2-wass},
\[
\Psi(\hat\nu)
= \Phi(\hat\pi)
\geq \sum_{s<t}\lambda_s\lambda_t \Wtwo^2(\mu^s, \mu^t) \geq \lambda_k \sumitoNlambdai \Wtwo^2(\mu^k, \mu^i)
\]
such that
\[
\frac{\Psi(\tilde\nu)}{\Psi(\hat\nu)}
\leq \frac 1 {\lambda_k} \leq N.
\]
Further, for the choice $\tilde\nu\coloneqq\sumitoNlambdai\mu_i$, setting $\pi^{st}\in\argmin_{\pi \in \Pi (\mu^s, \mu^t)} \langle c, \pi\rangle$, we get
\[
\Psi(\tilde\nu)
= \sum_{s=1}^N\lambda_s\Wtwo^2(\mu_s, \sum_{t=1}^N\lambda_t\mu^t)
\leq \sum_{s=1}^N\lambda_s \langle c, \sum_{t=1}^N\lambda_t \pi^{st}\rangle
= 2\sum_{s<t}\lambda_s\lambda_t\Wtwo^2(\mu^s, \mu^t)
\]
such that
\[
\frac{\Psi(\tilde\nu)}{\Psi(\hat\nu)}
\leq 2.
\]
However, these linear combinations of the input measures clearly do not convey any useful information for interpolating between the measures in a Wasserstein sense, and will be far from $\hat\nu$ in practice.

\subsection{Algorithm 1 -- Reference Algorithm}\label{sec:ref_alg}
In this subsection, we consider the cost function $c_r$ in \eqref{eq:cost_alg_I}.
First, we observe the following fact.
\begin{lemma}\label{lem:ref_wasserstein_step}
	Let 
	$\tilde\pi = \tilde\pi^{(N)} = \sum_{j=1}^{\tilde M} \tilde \pi_j \delta(\tilde x_{1,j}, \ldots, \tilde x_{N,j})$ 
	be obtained by \eqref{eq:ref_alg} with cost function \eqref{eq:cost_alg_I}. 
	Then it holds for $i=1,\ldots,N$ that
	\begin{equation} \label{eq:ref_lemma_cost}
	\sum_{j=1}^{\tilde M} \tilde \pi_j \| \tilde x_{1,j} - \tilde x_{i,j}\|_2^2
	= \langle c, (P_{1,i})_\# \tilde \pi \rangle
	= \Wtwo^2(\mu^1, \mu^i).
	\end{equation}
\end{lemma}
\begin{proof}
	The first equality follows by construction, so that we only have to verify the second one.
	We prove this assertion by induction on $r$, 
	i.e., we show  for $i=2,\ldots,r$ that
	\begin{equation}\label{eq:ref_lemma_cost_r}
	\langle c, (P_{1,i})_\# \tilde \pi^{(r)} \rangle  = \mathcal W_2 ^2(\mu^1, \mu^i).
	\end{equation}
	For $r=2$, the assertion follows by construction.
	
	Let $r \ge 3$ and assume that \eqref{eq:ref_lemma_cost_r} is fulfilled for $\tilde \pi^{(r-1)}$.
	For $i=2,\ldots,r-1$, we have by the marginal constraint in \eqref{eq:ref_alg} that
	$$
	(P_{1,i})_\# \tilde \pi^{(r)} = (P_{1,i})_\# \tilde \pi^{(r-1)}
	$$
	so that  by induction assumption $\tilde \pi^{(r)}$ fulfills the assertion for those $i$.
	It remains to consider $i=r$.
	Let
	$$
	\tilde \pi^{(r)} = \sum_{j=1}^{\tilde M_r} \tilde\pi_j^{(r)} \delta( \tilde x_{1,j}^{(r)},\ldots, \tilde x_{N,j}^{(r)} ).
	$$
	Then it follows
	\begin{align*}
	\sum_{j=1}^{\tilde M_r} \tilde \pi_j^{(r)} \Vert \tilde x_{1, j}^{(r)} - \tilde x_{r, j}^{(r)}\Vert^2
	= \big\langle c, (P_{1,r})_\# \tilde \pi^{(r)} \big \rangle
	\ge  \mathcal W _2 ^2(\mu^1,\mu^r).
	\end{align*}
	To show the reverse direction, let
	\[
	\pi^{1r} \in \argmin_{\pi\in \Pi(\mu^1, \mu^r)} \sum_{j_1, j_r} \pi_{j_1, j_r} \Vert x^1_{j_1} - x^r_{j_r} \Vert^2
	\]
	and define the measure 
	\[
	\bar \pi^r \coloneqq \sum_{j_1,\ldots,j_r} \bar \pi^r_{j_1, \dots, j_r} \delta(x^1_{j_1}, \ldots,x^r_{j_r}),
	\qquad
	\bar \pi^r_{j_1, \dots, j_r} \coloneqq \frac{\tilde\pi^{(r-1)}_{j_1, \dots, j_{r-1}}\pi^{1r}_{j_1, j_r}}{\mu_{j_1}^1}.
	\]
	We have that $\bar \pi^r\in  \Pi(\tilde\pi^{(r-1)}, \mu^r)$, since
	\[
	\sum_{j_r} \bar \pi^r_{j_1, \dots, j_r}
	= 
	\sum_{j_r} \frac{\tilde\pi^{(r-1)}_{j_1, \dots, j_{r-1}}\pi^{1r}_{j_1, j_r}}{\mu_{j_1}^1}
	= \frac{\tilde\pi^{(r-1)}_{j_1, \dots, j_{r-1}}}{\mu_{j_1}^1} \sum_{j_r} \pi^{1r}_{j_1, j_r}
	= \tilde\pi^{(r-1)}_{j_1, \dots, j_{r-1}}
	\]
	and
	\begin{align*}
	\sum_{j_1, \dots, j_{r-1}} \bar \pi^r_{j_1, \dots, j_r}
	&= \sum_{j_1, \dots, j_{r-1}} \frac{\tilde\pi^{(r-1)}_{j_1, \dots, j_{r-1}}\pi^{1r}_{j_1, j_r}}{\mu_{j_1}^1}
	= \sum_{j_1} \frac{\pi^{1r}_{j_1, j_r}}{\mu_{j_1}^1} \sum_{j_2, \dots, j_{r-1}}\tilde\pi^{(r-1)}_{j_1, \dots, j_{r-1}}\\
	&= \sum_{j_1} \pi^{1r}_{j_1, j_r} = \mu^r_{j_r}.
	\end{align*}
	Since $\tilde\pi^{(r)}\in \argmin_{\pi \in \Pi(\tilde\pi^{(r-1)}, \mu^r)} \langle c_r, \pi \rangle$ by construction, we obtain
	\begin{align*}
	\langle c, (P_{1, r})_\# \tilde\pi^{(r)} \rangle
	&= \langle c_r, \tilde\pi^{(r)} \rangle
	\leq \langle  c_r, \bar \pi^r\rangle 
	= \sum_{j_1, \dots, j_r} c_r(x_{j_1}^1, \dots, x_{j_r}^r) \bar \pi^r_{j_1, \dots, j_r} \\
	&= \sum_{j_1, j_r} \Vert x_{j_1}^1 - x_{j_r}^r \Vert^2 \sum_{j_2, \dots, j_{r-1}} \bar \pi^r_{j_1, \dots, j_r}
	=\; \sum_{j_1, j_r} \pi^{1r}_{j_1, j_r} \Vert x_{j_1}^1 - x_{j_r}^r\Vert^2
	= \Wtwo^2(\mu^1, \mu^r),
	\end{align*}
	which yields the assertion.
\end{proof}

Now we can give an intuition for a transport plan $\tilde\pi$ obtained from iteration \eqref{eq:ref_alg} with cost function \eqref{eq:cost_alg_I}.
Consider the multivariate cost function
\[
\tilde c(x_1, \dots, x_N)
\coloneqq \sum_{i=2}^N \lambda_i \Vert x_1-x_i\Vert^2
=  \sum_{i=2}^N \lambda_i c(x_1, x_i).
\]
Note that this is, up to the multiplicative constant $\lambda_1$, precisely $c_{\mathrm{MOT}}$ without all terms that do not depend on $x_1$.
For any $\pi\in \Pi(\mu^1, \dots, \mu^N)$ it holds
\[
\langle \tilde c, \pi\rangle
= \sum_{i=2}^N \lambda_i \langle c, (P_{1, i})_\#\pi\rangle
\geq \sum_{i=2}^N \lambda_i \Wtwo^2(\mu^1, \mu^i).
\]
As a consequence of Lemma~\ref{lem:ref_wasserstein_step}, equality holds for $\pi=\tilde\pi$, such that $\tilde\pi$ is the solution to the minimization problem
\[
\min_{\pi\in\Pi(\mu^1, \dots, \mu^N)} \langle \tilde c, \pi\rangle.
\]

Indeed, due to the special cost function, we can find $\tilde \pi^{(r)}$ by solving the $N$ optimal transport
problems belonging to $\mathcal W_2^2(\mu^1,\mu^i)$, $i=1,\ldots,N$ and then choosing any way to fit the marginals 
$\Pi(\tilde \pi^{(r-1)},\mu^r)$. Algorithm \ref{alg:ref} shows a possibility how this can be done such that
the resulting measure $\tilde \pi = \tilde \pi^{(N)}$ has a small support.
Note that the first for-loop, which is the computational bottleneck of the algorithm, can readily be parallelized for further speedups.

\begin{algorithm}[htb]
	\begin{algorithmic}
		\State \textbf{Input:} Discrete measures $\mu^i = \sum_{j=1}^{n_i} \mu^i_j \delta(x^i_{j})$, $i=1, \dots, N$, with $x^1_1<\dots<x^1_{n_1}$ if $d=1$
		\For{$i=2,\dots, N$}
		\State Compute 
		\begin{align}
		\pi^i
		&\in
		\argmin_{\pi\in\Pi(\mu^1, \mu^i)} \langle c, \pi \rangle
		= \sum_j \pi^i_j (x_{1, i, j}, x_{i,j})
		\quad \text{s.t.} \quad \#\supp(\pi)\leq n_1 + n_i -1
		\end{align}
		\If{$d=1$}
		\State Sort lexicographically $(x_{1, i, 1}, x_{i, 1}) < (x_{1, i, 2}, x_{i, 2}) < \dots$
		\EndIf
		\EndFor
		\State \textbf{Initialization:} $\tilde\pi = 0$
		\For{$k=1, \dots, n_1$}
		\While{$x_k^1\in P_1(\supp(\pi^i))$ for $i=2, \dots, N$}
		\For{$i=2, \dots, N$}
		\State $j_i \gets \min\{j:x_{1, i, j} = x_k^1\}$
		\EndFor
		\State $h\gets \min_{j_i} \pi^i_{j_i}$
		\State $\tilde\pi \gets \tilde\pi + h \delta(x_{k}^1, x_{2, j_2}, \dots, x_{N, j_N})$ 
		\For{$i=2, \dots, N$}
		\State $\pi^i \gets \pi^i - h \delta (x_k^1, x_{i, j_i})$
		\EndFor
		\EndWhile		
		\EndFor
		\State \textbf{Output:} $\tilde\pi$
		\caption{Reference algorithm}
		\label{alg:ref}
	\end{algorithmic}
\end{algorithm}
Note that for $N=2$, the choice $j_i \gets \min\{j:x_{1, i, j} = x_k^1\}$ in Algorithm~\ref{alg:ref} corresponds to the so-called north-west corner rule, see e.g. \cite[Sec.~3.4.2]{PC19book}.
In that case, this heuristic is often used to merely produce a sparse coupling as an initialization for another optimization procedure such as the network simplex method.
Despite there being more elaborate approaches, we use this one for simplicity.

Here is an example of how the algorithm performs:
\begin{example}\label{ex:ot_plans_cannot_be_read_off_of_mot}
	Suppose
	\begin{align*}
	\mu^1 = \frac 1 2 \delta(x_1^1) + \frac 1 2 \delta(x_2^1), \quad
	\mu^2 = \frac 1 4 \delta(x_1^2) + \frac 3 4 \delta(x_2^2), \quad 
	\mu^3 = \frac 1 3 \delta(x_1^3) + \frac 2 3 \delta(x_2^3)
	\end{align*}
	and that
	\begin{align*}
	\pi^2 &= \frac 1 4 \delta(x_1^1, x_1^2) + \frac 1 4 \delta(x_1^1, x_2^2) + \frac 1 2 \delta(x_2^1, x_2^2)
	=  \frac 1 4 \delta(x_{1, 2, 1}, x_{2, 1}) + \frac 1 4 \delta(x_{1, 2, 2}, x_{2, 2}) + \frac 1 2 \delta(x_{1, 2, 3}, x_{2, 3}) \\ 
	\pi^3 &= \frac 1 3 \delta(x_1^1, x_1^3) + \frac 1 6 \delta(x_1^1, x_2^3) + \frac 1 2 \delta(x_2^1, x_2^3)
	= \frac 1 3 \delta(x_{1, 3, 1}, x_{3, 1}) + \frac 1 6 \delta(x_{1, 3, 2}, x_{3, 2}) + \frac 1 2 \delta(x_{1, 3, 3}, x_{3, 3}).
	\end{align*}
	Initialize $\tilde\pi\gets 0$.
	
	$\boldsymbol{k = 1}$
	
	$j_2 \gets 1$, $j_3 \gets 1 \Rightarrow h\gets \frac 1 4$. Then
	\begin{align*}
	\tilde\pi &\gets \frac 1 4 \delta(x_1^1, x_1^2, x_1^3), \\
	\pi^2 &\gets \frac 1 4 \delta(x_1^1, x_2^2) + \frac 1 2 \delta(x_2^1, x_2^2), \quad
	\pi^3 \gets \frac 1 {12} \delta(x_1^1, x_1^3) + \frac 1 6 \delta(x_1^1, x_1^3) + \frac 1 2 \delta(x_2^1, x_2^3).
	\end{align*}
	$j_2 \gets 2$, $j_3 \gets 1 \Rightarrow h\gets \frac 1 {12}$. Then
	\begin{align*}
	\tilde\pi &\gets \frac 1 4 \delta(x_1^1, x_1^2, x_1^3) + \frac 1 {12} \delta(x_1^1, x_2^2, x_1^3), \\
	\pi^2 &\gets \frac 1 6 \delta(x_1^1, x_2^2) + \frac 1 2 \delta(x_2^1, x_2^2), \quad
	\pi^3 \gets \frac 1 6 \delta(x_1^1, x_1^3) + \frac 1 2 \delta(x_2^1, x_2^3).
	\end{align*}
	$j_2 \gets 2$, $j_3 \gets 2 \Rightarrow h\gets \frac 1 6$. Then
	\begin{align*}
	\tilde\pi &\gets \frac 1 4 \delta(x_1^1, x_1^2, x_1^3) + \frac 1 {12} \delta(x_1^1, x_2^2, x_1^3) + \frac 1 6 \delta(x_1^1, x_2^2, x_1^3), \\
	\pi^2 &\gets \frac 1 2 \delta(x_2^1, x_2^2), \quad
	\pi^3 \gets \frac 1 2 \delta(x_2^1, x_2^3).
	\end{align*}
	
	$\boldsymbol{k = 2}$
	
	$j_2 \gets 3$, $j_3 \gets 3 \Rightarrow h\gets \frac 1 2$. Then
	\begin{align*}
	\tilde\pi &\gets \frac 1 4 \delta(x_1^1, x_1^2, x_1^3) + \frac 1 {12} \delta(x_1^1, x_2^2, x_1^3) + \frac 1 6 \delta(x_1^1, x_2^2, x_1^3)+ \frac 1 2 \delta(x_2^1, x_2^2, x_2^3), \\
	\pi^2 &\gets 0, \quad
	\pi^3 \gets 0.
	\end{align*}
\end{example}

By the next proposition, $\tilde \pi$ has a sparse support.

\begin{proposition}
	Let $\tilde\pi$ be computed by Algorithm~\ref{alg:ref}. Then
	\[
	\#\supp(\tilde\pi) \leq \sumitoN n_i - N + 1.
	\]
\end{proposition}

\begin{proof}
	Let $n(i, k)\coloneqq \# \{ j_i: (x_k^1, x_{i, j_i}) \in \supp(\pi^i) \}$ 
	denote the number of support pairs in $\pi^i$ that contain $x_k^1$.
	Then, for each $k =1, \ldots,n_1$ in Algorithm~\ref{alg:ref}, the while-loop is iterated at most
	\[
	\sum_{i=2}^N (n(i, k)-1)+1 = \sum_{i=2}^N n(i, k) - N + 2
	\]
	times. 	In total, this amounts to 
	\begin{align*}
	\sum_{k=1}^{n_1} \Big(\sum_{i=2}^N n(i, k) -N+2 \Big)
	&= n_1(2-N) + \sum_{i=2}^N \#\supp(\pi^i) 
	\\
	&\leq-(N-2)n_1 + \sum_{i=2}^N (n_1 + n_i-1)
	= \sumitoN n_i - N + 1
	\end{align*}
	iterations.
	Noticing that there is exactly one support point added to $\tilde\pi$ in each while-iteration,  this yields the assumption.
\end{proof}

The following theorem gives an upper bound for the relative error.
Later we will see that this bound can in general not be improved.

\begin{theorem}\label{thm:ref-upper}
	Let $\tilde\pi$ be a multi-marginal plan obtained by \eqref{eq:ref_alg}.
	Then it holds
	\begin{equation}\label{eq:ref_worst_case}
	\frac{\Phi(\tilde\pi)}{\Phi(\hat\pi)} \leq \frac 1 {\lambda_1},
	\end{equation}
	If we choose reference measure $\mu^1$ according to $\lambda_1 \ge  \dots \ge \lambda_N$, then the right-hand side is not larger than $N$.
	If we choose the reference measure randomly, so that with probability $\lambda_i$ we take $\mu^i$, $i=1,\ldots,N$ as reference measure, 
	then it holds
	\begin{eqnarray}\label{eq:ref_rand_bound}
	\frac{\mathbb E [\Phi(\tilde \pi)]}{\Phi(\hat\pi)} \leq 2.
	\end{eqnarray}
\end{theorem}

\begin{proof}
	By Lemma~\ref{lem:ref_wasserstein_step} we know for 
	$\tilde \pi = \sum_{j=1}^{\tilde M} \tilde \pi_j \delta(\tilde x_{1,j}, \ldots,\tilde x_{N,j})$ that
	\begin{align*}
	\sum_{j=1}^{\tilde M} \tilde \pi_j \Vert \tilde x_{1, j} - \tilde x_{i, j} \Vert^2
	= \Wtwo^2(\mu^1, \mu^i).
	\end{align*}
	Then, for $\tilde m_j \coloneqq M_\lambda(\tilde x_{1, j}, \dots, \tilde x_{N, j})$, we obtain by \eqref{eq:Phi_2} and Lemma~\ref{lem:jensen_eq}
	that
	\begin{align*}
	\Phi(\tilde\pi) 
	&= 
	\sum_{j=1}^{\tilde M} \tilde \pi_j \sumitoNlambdai \Vert \tilde x_{i, j} - \tilde m_j\Vert^2
	= 
	\sum_{j=1}^{\tilde M} \tilde \pi_j \Big( -\Vert \tilde x_{1, j} -\tilde m_j\Vert^2 + \sumitoNlambdai \Vert \tilde x_{1, j} - \tilde x_{i, j}\Vert^2 \Big) \\
	&\leq \sumitoNlambdai \sum_{j=1}^{\tilde M}\tilde \pi_j  \Vert \tilde x_{1, j} - \tilde x_{i, j}\Vert^2
	= \sum_{i=2}^N \lambda_i \sum_{j=1}^{\tilde M}\tilde \pi_j  \Vert \tilde x_{1, j} - \tilde x_{i, j}\Vert^2
	= \sum_{i=2}^N \lambda_i \Wtwo^2(\mu^1, \mu^i).
	\end{align*}
	On the other hand, we conclude by \eqref{eq:2-wass} that
	\begin{align*}
	\Phi(\hat\pi)
	\ge \sum_{s<t} \lambda_s\lambda_t \Wtwo^2(\mu^s, \mu^t) \geq \lambda_1 \sum_{i=2}^N \lambda_i \Wtwo^2(\mu^1, \mu^i).
	\end{align*}
	Thus,
	\[
	\frac{\Phi(\tilde\pi)}{\Phi(\hat\pi)} 
	\leq 
	\frac{\sum_{i=2}^N \lambda_i \Wtwo^2(\mu^1, \mu^i)}{\lambda_1 \sum_{i=2}^N \lambda_i \Wtwo^2(\mu^1, \mu^i)} 
	= \frac 1 {\lambda_1}.
	\]
	Further, if the reference measure is chosen randomly according to the probabilities $\lambda_i$, we obtain
	\[
	\mathbb E [ \Phi(\tilde\pi) ]
	\leq \sum_{i=1}^N \lambda_{i} \sum_{i\neq j} \lambda_{j} \Wtwo^2(\mu^i, \mu^j)
	= 
	2 \sum_{i<j} \lambda_{i} \lambda_{j} \Wtwo^2(\mu^i, \mu^j),
	\]
	so that
	\[
	\frac{\mathbb E [\Phi(\tilde \pi)]}{\Phi(\hat\pi)} \leq 2.
	\]
\end{proof}

\begin{theorem} \label{thm:ref-lower}
	For every odd $N\in \mathbb{N}$ and $\varepsilon > 0$, there exist measures $\mu_1, \dots, \mu_N$, $\lambda\in \Delta_N$, such that
	\begin{equation}\label{eq:ref_worst_example_odd}
	\frac{\Phi(\tilde\pi)}{\Phi(\hat\pi)} \geq N-\varepsilon.
	\end{equation}
	For every even $N\in \mathbb N$, there exist measures $\mu_1, \dots, \mu_N$, $\lambda\in \Delta_N$, such that
	\[
	\frac{\Phi(\tilde\pi)}{\Phi(\hat\pi)} \geq N-\frac 1 {N-1} - \varepsilon.
	\]
	Further, if the reference measure is chosen randomly as in Theorem~\ref{thm:ref-upper}, the upper bound \eqref{eq:ref_rand_bound} is asymptotically tight, meaning that for every $N\in \mathbb N$, 
	there are measures $\mu_1, \dots, \mu_N$, $\lambda\in \Delta_N$, 
	such that if the corresponding plans for every $N$ are denoted by $\tilde\pi(N)$ and $\hat\pi(N)$, we have
	\[
	\lim_{N\to\infty} \frac{\mathbb E [\Phi(\tilde \pi(N))]}{\Phi(\hat\pi(N))} = 2.
	\]
\end{theorem}
The proof can be found in the appendix.

\subsection{Algorithm 2 -- Greedy Algorithm}\label{sec:greedy}

The second algorithm with cost function $c_r$ in \eqref{eq:cost_alg_II} is a greedy algorithm.
Again it requires only the solution of a two-marginal Wasserstein problem in each iteration.
To this end,
let 
\[
\tilde \pi^{(r-1)} = \sum_{k=1}^{\tilde M_{r-1}} \tilde \pi^{(r-1)}_k \delta (\tilde x_{1,k}^{(r-1)}, \ldots, \tilde x_{r-1,k}^{(r-1)} ).
\]
Then the iteration \eqref{eq:ref_alg} can be rewritten as in Algorithm \ref{alg:greedy}.

\begin{algorithm}[htb]
	\begin{algorithmic}
		\State \textbf{Input:} Discrete measures $\mu^1, \dots, \mu^N$, weights $\lambda \in \Delta_N$
		\State \textbf{Initialization:} $\tilde\pi^{(1)} \coloneqq \mu^1$
		\For {$r=2, \dots, N$} 
		\begin{align*}
		\tilde \nu^{(r-1)} 
		&\coloneqq 
		(M_{\bar \lambda_{r-1}} )_\# \tilde \pi^{(r-1)} = 
		\sum_{k=1}^{\tilde M_{r-1}} \tilde \pi^{(r-1)}_k \delta (\tilde m_k^{(r-1)} ),
		\\
		&\qquad \text{where } \tilde m_k^{(r-1)} 
		\coloneqq 
		M_{\bar \lambda_{r-1}}(\tilde x_{1,k}^{(r-1)}, \ldots, \tilde x_{r-1,k}^{(r-1)} )
		= \sum_{i=1}^{r-1} \bar \lambda_{i, r-1} \tilde x_{i,k}^{(r-1)}.
		\\
		\bar \pi^{(r)}
		&\in \argmin_{\pi \in \Pi (\tilde \nu^{(r-1)}) , \mu^r)} \langle c , \pi \rangle 
		= \sum_{j=1}^{\tilde M_{r}} \bar \pi^{(r)}_j \delta ( \bar m_j^{(r-1)}, \tilde x_{r,j}^{(r)} )
		\quad 
		\mathrm{s.t.} 
		\quad \#\supp(\bar \pi^{(r)}) \leq \tilde M_{r-1} + n_r -1,
		\\
		& \qquad\text{where for each } j,\; \bar m_j^{(r-1)} = \tilde m_k^{(r-1)} \text{ for some } k \in \{1, \dots, \tilde M_{r-1} \} \\
		& \qquad\text{and } \tilde x_{r, j}^{(r)} = x_l^r\in \supp(\mu^r) \text{ for some } l\in \{ 1, \dots, n_r \} \\
		\tilde \pi ^{(r)} 
		&\coloneqq 
		\sum_{j=1}^{\tilde M_{r}} \bar \pi^{(r)}_j 
		\delta ( \tilde x_{1,j}^{(r)}, \dots, \tilde x_{r,j}^{(r)}), \quad\text{where } (\tilde x_{1,j}^{(r)}, \dots, \tilde x_{r-1,j}^{(r)}) \text{ corresponds to } \bar m_j^{(r-1)} \\
		&\qquad \text{ via } M_{\bar\lambda_{r-1}}(\tilde x_{1,j}^{(r)}, \dots, \tilde x_{r-1,j}^{(r)}) = \bar m_j^{(r-1)}
		\end{align*}
		\EndFor
		\State \textbf{Output:} $\tilde \pi \coloneqq \tilde\pi^{(N)}$
		\caption{Greedy algorithm}
		\label{alg:greedy}
	\end{algorithmic}
\end{algorithm}

Since 
$$
\bar \lambda_{i,r} = (1- \bar \lambda_{r,r} ) \bar \lambda_{i,r-1}, \quad i=1,\ldots,r-1
$$
we get
\begin{align*}
\tilde \nu^{(r)}
&= 
\sum_{j=1}^{\tilde M_{r}} \bar \pi^{(r)}_j 
\delta ( \bar \lambda_{1,r} \tilde x_{1,j}^{(r)}+ \dots+ \bar \lambda_{r-1,r}\tilde x_{r-1,j}^{(r)} 
+  \bar \lambda_{r,r} \tilde x_{r,j}^{(r)} )\\
&= 
\sum_{j=1}^{\tilde M_{r}} \bar \pi^{(r)}_j 
\delta ( (1- \bar \lambda_{r,r}) \bar \lambda_{1,r-1} \tilde x_{1,j}^{(r)}+ \dots
+ (1- \bar \lambda_{r,r})\bar \lambda_{r-1,r-1}\tilde x_{r-1,j}^{(r)} 
+  \bar \lambda_{r,r}\tilde x_{r,j}^{(r)} )		\\
&= 
\sum_{j=1}^{\tilde M_{r}} \bar \pi^{(r)}_j 
\delta ( (1- \bar \lambda_{r,r}) \bar m_j^{(r-1)} 
+  \bar \lambda_{r,r}\tilde x_{r,j}^{(r)} ).
\end{align*}
In other words, we alternate computing a two-marginal OT plan between $\tilde\nu^{(r-1)}$ and $\mu^r$ and an appropriately weighted so-called McCann-interpolation between them, which is optimal for two measures.
Otherwise, only some bookkeeping of indices is required to obtain an approximate MOT plan.

First, we show that the proposed procedure produces a plan with sparse support. 

\begin{proposition}
	Let $\tilde\pi$ be computed by Algorithm~\ref{alg:greedy}. 
	Then it holds
	\[
	\#\supp(\tilde\pi) \leq \sumitoN n_i- N + 1.
	\]
\end{proposition}

\begin{proof}
	By construction, we have
	\[
	\#\supp (\tilde \pi^{(r)})
	= \# \supp(\bar\pi^{(r)})
	\leq \tilde M_{r-1} + n_r -1
	= \#\supp (\tilde \pi^{(r-1)}) + n_r - 1
	\]
	for all $r=2, \dots, N$.
	Thus, we get inductively
	\[
	\#\supp (\tilde \pi)
	\leq \#\supp (\tilde \pi^{(N-1)}) + n_N - 1
	\leq \sum_{i=1}^{N-1} n_{i} - (N-1) + 1 + n_N - 1
	= \sum_{i=1}^N n_i - N +1.
	\]
\end{proof}

Again, we would like to analyze the relative error $\Phi(\tilde\pi) / \Phi(\hat\pi)$.
We start with an upper bound.

\begin{theorem}\label{thm:greedy-upper}
	Let $\lambda_1\geq \dots \geq \lambda_N$ and
	$\tilde \pi$ be computed by Algorithm~\ref{alg:greedy}.
	Then it holds
	\begin{equation}\label{eq:greedy_upper}
	\frac{\Phi(\tilde\pi)}{\Phi(\hat\pi)} \leq \frac 1 3 (2N^2 - 5).
	\end{equation}
\end{theorem}
A linear lower bound is given in the next theorem.

\begin{theorem}\label{thm:greedy-lower}
	For each $N\geq 2$,
	and $\varepsilon>0$,
	there exist measures $\mu^1, \dots, \mu^N$ and $\lambda\in \Delta_N$, such that
	\[
	\frac{\Phi(\tilde\pi)}{\Phi(\hat\pi)}
	\geq \frac{N-H_N}{\frac{\pi^2}{6}+1} -\varepsilon
	\geq \frac 1 4 N - \frac 1 3,
	\]
	where $\tilde\pi$ denotes a plan computed by Algorithm~\ref{alg:greedy} and $H_N=\sumitoN 1/i$.
\end{theorem}

Finally, similar as in Section~\ref{sec:ref_alg}, we show that for $\lambda\equiv 1/N$ it is possible to reduce the worst case error by randomizing the order of the input measures.
\begin{theorem}\label{thm:greedy-upper-rand}
	Let $\mathbb E[\Phi(\tilde\pi)]$ be the expected value of the costs of Algorithm~\ref{alg:greedy} when it is run on the inputs $\lambda\equiv 1/N$ and $\mu^{\sigma(1)}, \dots, \mu^{\sigma(N)}$, where $\sigma$ is a permutation of length $N$ chosen uniformly at random.
	Then it holds for all $N\geq 2$, that
	\begin{equation}\label{eq:greedy_upper-rand}
	\frac{\mathbb E[\Phi(\tilde\pi)]}{\Phi(\hat\pi)}
	\leq \frac 1 {12}(11N-4-\frac 6 {N-1}) < N.
	\end{equation}
\end{theorem}

The proofs can be found in the appendix.
Although the given upper bounds \eqref{eq:greedy_upper} and \eqref{eq:greedy_upper-rand} are quadratic and linear in $N$, we conjecture that they can be improved.

Although in our experience, the greedy algorithm usually performs better than the reference algorithm, see e.g. Section~\ref{sec:ellipses}, we conclude this section by giving an example of $N=4$ two-point-measures, where one algorithm or the other is better, only depending on the order of the input measures.
\begin{example}
	Consider
	\[
	x_1 = (-1, \frac 5 8),\quad
	x_2 = (0, \frac 5 8),\quad
	x_3 = (1, \frac 5 8),\quad
	x_4 = -x_1,\quad
	x_5 = -x_2,\quad
	x_6=-x_3
	\]
	and
	\[
	\nu^1 = \frac 1 2 (\delta(x_1)+\delta(x_4)), \quad
	\nu^2 = \frac 1 2 (\delta(x_2)+\delta(x_5)), \quad
	\nu^3 = \frac 1 2 (\delta(x_3)+\delta(x_6)),
	\]
	with $\lambda\equiv \frac 1 4$.
	\begin{figure}[htb]
		\renewcommand\curwidth{.7\textwidth}
		\centering
		\includegraphics[width=\curwidth]{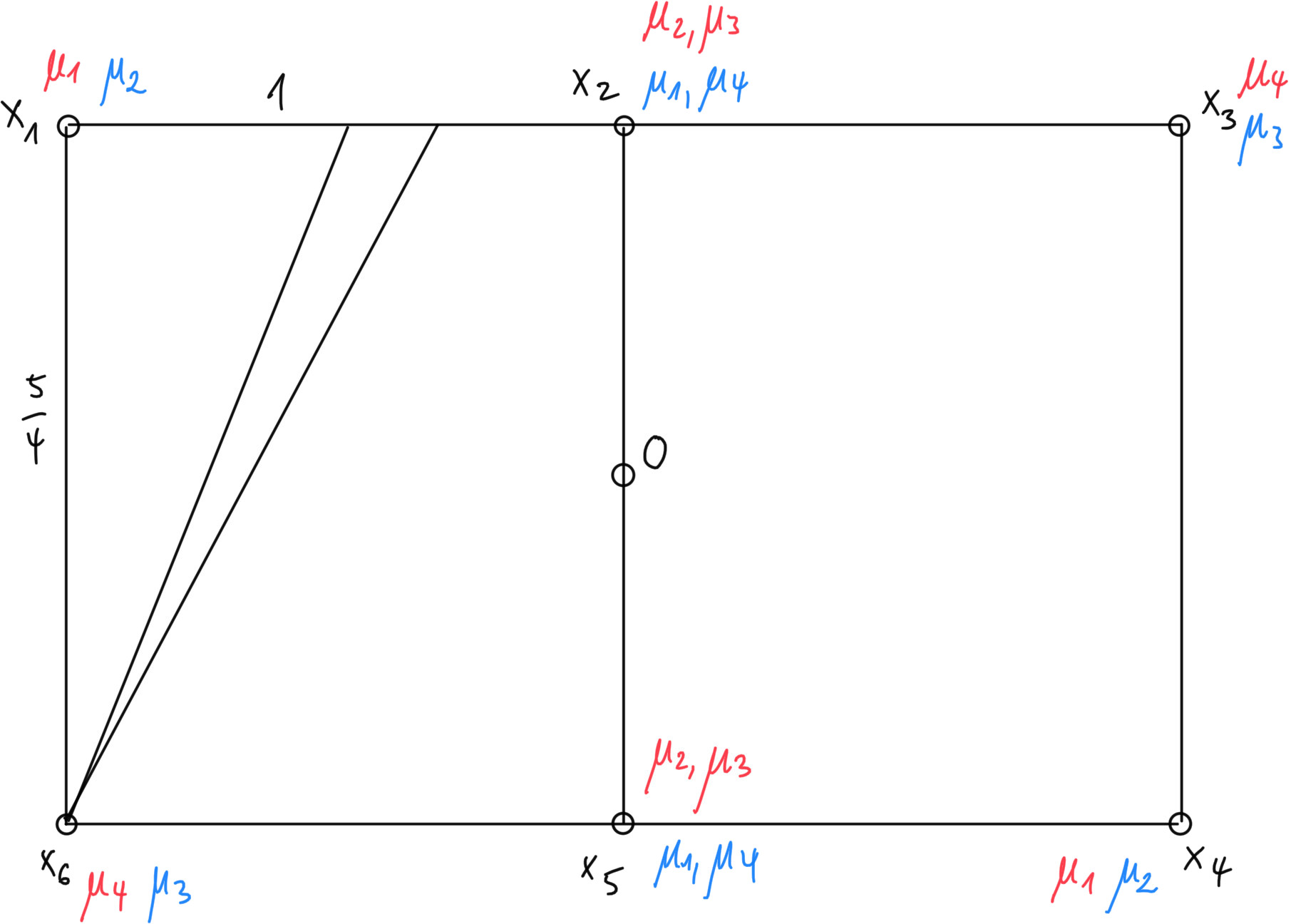}
		\caption{Sketch of the example showing that neither Algorithm~\ref{alg:ref} nor Algorithm~\ref{alg:greedy} is always better. Next to the names of the support points are written the measures to whose support they belong. The two different colors indicate two different orderings. The greedy algorithm performs better for the measures written in red, whereas the reference algorithm performs better for the measures written in blue. The diagonal lines indicate the distances on the left hand side in \eqref{eq:neither_better_step_1} and \eqref{eq:neither_better_step_2}, respectively.}
		\label{fig:neither_is_better}
	\end{figure}
	By elementary computations, for $(\mu_1, \mu_2, \mu_3, \mu_4) = (\nu_1, \nu_2, \nu_2, \nu_3)$, the optimal plan is
	\[
	\hat\pi = \frac 1 2 (\delta(x_1, x_2, x_2, x_3) + \delta(x_4, x_5, x_5, x_6)),
	\]
	whereas for
	\[
	\tilde\pi = \frac 1 2 (\delta(x_1, x_2, x_2, x_6) + \delta(x_4, x_5, x_5, x_3)),
	\]
	we have $\Phi(\tilde\pi) > \Phi(\hat\pi)$.
	Further computations show that
	\begin{align}
		\Vert \frac 1 2 x_1 + \frac 1 2 x_2 - x_6\Vert^2 &< \Vert \frac 1 2 x_1 + \frac 1 2 x_2 - x_3\Vert^2, \quad \text{but} \label{eq:neither_better_step_1}\\
		\Vert \frac 1 3 x_1 + \frac 2 3 x_2 - x_6\Vert^2 &> \Vert \frac 1 3 x_1 + \frac 2 3 x_2 - x_3\Vert^2 \label{eq:neither_better_step_2}
	\end{align}
	such that by definition of the algorithms, for $(\mu_1, \mu_2, \mu_3, \mu_4) = (\nu_1, \nu_2, \nu_2, \nu_3)$, we have $\pi^{\text{greedy}}=\hat\pi$ and $\pi^{\text{ref}}=\tilde\pi$, whereas for the ordering $(\mu_1, \mu_2, \mu_3, \mu_4) = (\nu_2, \nu_1, \nu_3, \nu_2)$, we get $\pi^{\text{greedy}}=\tilde\pi$ and $\pi^{\text{ref}}=\hat\pi$.
	For an illustration of this example, see Figure~\ref{fig:neither_is_better}.
\end{example}

\subsection{Optimality for $d=1$}

In this section, we assume $d=1$, i.e., $\mu^1, \dots, \mu^N\in \mathcal P(\R)$.
Then we can show that the presented algorithms yield the optimal solution $\hat\pi$.
To this end, we first make the following definitions to characterize $\hat\pi$.
\begin{definition}
	For $x_k = (x_{1, k}, \dots, x_{N, k})$, $x_l = (x_{1, l}, \dots, x_{N, l})$, define the partial order
	\[
	x_k \preceq x_l \ratio\Leftrightarrow x_{i, k} \leq x_{i, l} \quad \text{for all} \quad i=1, \dots, N.
	\]
	A multi-marginal plan $\pi\in \Pi(\mu^1, \dots, \mu^N)$ is said to have the \emph{sorting property}, if $\preceq$ is a total order on $\supp(\pi)$.
\end{definition}
Next, we show that the sorting property is equivalent to optimality, see also \cite[Remark~9.6]{PC19book}.
This is a well-known result in the case $N=2$, see e.g. \cite[Lemma~2.8, Theorem~2.9]{S15otapplied}.
\begin{proposition}\label{prop:d1_characterization}
	A plan $\pi\in \Pi(\mu^1, \dots, \mu^N)$ has the sorting property if and only if it is optimal.
\end{proposition}
\begin{proof}
	Suppose $\pi\in \Pi(\mu^1, \dots, \mu^N)$ has the sorting property.
	Write
	\[
	\pi = \sum_{j=1}^M \pi_j \delta(x_{1, j}, \dots, x_{N, j}).
	\]
	Then for any $s<t$,
	\[
	\pi^{st} \coloneqq \sum_{j=1}^M \pi_j \delta(x_{s, j}, x_{t, j}) = (P_{s, t})_\# \pi
	\]
	clearly has the sorting property as well.
	For any $(x, y), (x', y')\in \pi^{st}$ with $x<x'$, it follows that $(x, y)\preceq (x', y')$, otherwise $\preceq$ would not be a total ordering.
	Thus, $x<x' \Rightarrow y\leq y'$, such that the condition from \cite[Lemma~2.8]{S15otapplied} is fulfilled and $\pi^{st}$ is optimal by \cite[Theorem~2.9]{S15otapplied}.
	Hence,
	\[
	\Phi(\pi)
	= \sum_{s<t}^N \lambda_s\lambda_t \sum_{j=1}^M \pi_j \Vert x_{s, j} - x_{t, j}\Vert^2
	= \sum_{s<t}^N \lambda_s\lambda_t \Wtwo^2(\mu^s, \mu^t)
	\leq \Phi(\hat\pi)
	\]
	by \eqref{eq:2-wass}, which shows optimality of $\pi$.
	
	The other direction is shown by contradiction: Let $\hat\pi$ be optimal and assume that there exist two tuples $\hat x_k$, $\hat x_l\in \supp(\hat\pi)$ with $\hat x_k\npreceq \hat x_l$ and $\hat x_l\npreceq \hat x_k$.
	That is, there exist $1\leq s, t\leq N$ such that $\hat x_{s, k} > \hat x_{s, l}$ and $\hat x_{t, l} > \hat x_{t, k}$.
	Set $\hat m_j = \sumitoNlambdai \hat x_{i, j}$ for all $j$ and assume without loss of generality that $\hat m_k\leq \hat m_l$ and $\hat x_{N, k}>\hat x_{N, l}$.
	Let $h\coloneqq \min(\hat\pi_k, \hat\pi_l)$ and define the coupling
	\[
	\pi' \coloneqq \hat\pi + h(\delta(x_k') + \delta(x_l') - \delta(\hat x_k) - \delta(\hat x_l)),
	\]
	where
	\begin{align*}
		x_k' &\coloneqq (x_{1, k}', \dots, x_{N-1, k}', x_{N, k}') \coloneqq (\hat x_{1, k}, \dots, \hat x_{N-1, k}, \hat x_{N, l}) \\
		x_l' &\coloneqq (x_{1, l}', \dots, x_{N-1, l}', x_{N, l}') \coloneqq (\hat x_{1, l}, \dots, \hat x_{N-1, l}, \hat x_{N, k}) \\
		m_j' &\coloneqq \sumitoNlambdai x_{i, j}', \quad j=k, l.
	\end{align*}
	Clearly, $\pi'$ fulfills the same marginal constraints as $\hat\pi$.
	By Lemma~\ref{lem:jensen_eq} and since $(\hat m_k - m_k')^2 = \lambda_N^2(\hat x_{N, k} - \hat x_{N, l})^2 > 0$ by assumption, it holds
	\[
	\sumitoNlambdai(x_{i, k}' - m_k')^2
	= -(\hat m_k-m_k')^2 + \sumitoNlambdai(x_{i, k}' - \hat m_k)^2
	< \sumitoNlambdai(x_{i, k}' - \hat m_k)^2,
	\]
	similarly for $l$.
	Thus, with \eqref{eq:Phi_2}, it holds
	\begin{align*}
	\frac 1 h (\Phi(\hat\pi) - \Phi(\pi')) &= \sumitoNlambdai (
	(\hat x_{i, k} - \hat m_k)^2
	+ (\hat x_{i, l} - \hat m_l)^2
	- (x_{i, k}' - m_k')^2
	- (x_{i, l}' - m_l')^2
	) \\
	&> \sumitoNlambdai (
	(\hat x_{i, k} - \hat m_k)^2
	+ (\hat x_{i, l} - \hat m_l)^2
	- (x_{i, k}' - \hat m_k)^2
	- (x_{i, l}' - \hat m_l)^2
	) \\
	&= \lambda_N(
	(\hat x_{N, k} - \hat m_k)^2
	+ (\hat x_{N, l} - \hat m_l)^2
	- (\hat x_{N, l} - \hat m_k)^2
	- (\hat x_{N, k} - \hat m_l)^2
	) \\
	&= -2\lambda_N(
	\hat x_{N, k}\hat m_k
	+ \hat x_{N, l}\hat m_l
	- \hat x_{N, l}\hat m_k
	- \hat x_{N, k}\hat m_l)
	\\
	&= -2\lambda_N\underbrace{(\hat x_{N, k} - \hat x_{N, l})}_{> 0}\underbrace{(\hat m_k-\hat m_l)}_{\leq 0} \geq 0,
	\end{align*}
	which contradicts the optimality of $\hat\pi$.
\end{proof}

\begin{proposition} If $d=1$, then $\Phi(\tilde\pi) = \Phi(\hat\pi)$, where $\tilde\pi$ is computed by Algorithm~\ref{alg:ref} or \ref{alg:greedy}.
\end{proposition}
\begin{proof}
	By Proposition~\ref{prop:d1_characterization}, it is enough to verify the sorting property of $\tilde\pi$ for each algorithm.
	
	Suppose that $\tilde\pi$ has been computed by Algorithm~\ref{alg:ref}.
	We show the assertion by induction over the number of while-loop-iterations adding to the support of $\tilde\pi$, which we denote by $l$.
	If $l=1$, there is nothing to show.
	Otherwise, assume as the induction hypothesis that
	\[
	(\tilde x_{1, 1}, \dots, \tilde x_{N, 1}) \preceq \dots \preceq (\tilde x_{1, l-1}, \dots, \tilde x_{N, l-1}).
	\]
	By definition of Algorithm~\ref{alg:ref}, $\tilde x_{1, 1}, \dots, \tilde x_{1, l-1}\leq \tilde x_{1, l}$.
	Further, since $\pi^i$ has the sorting property for all $i=2, \dots, N$ by  Proposition~\ref{prop:d1_characterization}, the lexicographical sorting
	\[
	(x_{1, i, 1}, x_{i, 1}) < (x_{1, i, 2}, x_{i, 2}) < \dots
	\]
	together with the choice
	\[
	j_i \gets \min\{ j: x_{1, i, j} = x^1_k \}
	\]
	in the previous iterations of Algorithm~\ref{alg:ref} ensures that we must have $\tilde x_{i, 1}, \dots, \tilde x_{i, l-1}\leq \tilde x_{i, l}$ for all $i=2, \dots, N$ as well.
	Thus,
	\[
	(\tilde x_{1, 1}, \dots, \tilde x_{N, 1}) \preceq \dots \preceq (\tilde x_{1, l}, \dots, \tilde x_{N, l}),
	\]
	completing the induction.
	
	Next, suppose $\tilde\pi$ has been computed by Algorithm~\ref{alg:greedy}.
	We show the statement by induction over $N$.
	The case $N=2$ is clear.
	For $N>2$, take $\tilde x_k = (\tilde x_{1, k}, \dots,\tilde x_{N, k})$, $\tilde x_l = (\tilde x_{1, l}, \dots, \tilde x_{N, l}) \in \supp(\tilde\pi)$.
	If $(\tilde x_{1, k}, \dots,\tilde x_{N-1, k}) = (\tilde x_{1, l}, \dots, \tilde x_{N-1, l})$, then either $\tilde x_k \preceq \tilde x_l$ or $\tilde x_l \preceq \tilde x_k$ follows directly.
	Otherwise, write $\tilde m^{(N-1)}_j = \sum_{i=1}^{N-1}\bar\lambda_{i, N-1}\tilde x_{i, j}$, $j=k, l$.
	Then we have $\tilde m^{(N-1)}_k \neq \tilde m^{(N-1)}_l$, since then $\tilde x_{i, k}\neq \tilde x_{i, l}$ for at least one $i\in \{1, \dots, N-1\}$ and by induction hypothesis it holds $\tilde x_{i, k}\leq \tilde x_{i, l}$ for all $i=1, \dots, N-1$ or $\tilde x_{i, k}\geq \tilde x_{i, l}$ for all $i=1, \dots, N-1$.
	Assume without loss of generality that $\tilde x_{i, k}\leq \tilde x_{i, l}$ for all $i=1, \dots, N-1$ so that $\tilde m^{(N-1)}_k < \tilde m^{(N-1)}_l$.
	It suffices to show that $\tilde x_{N, k}\leq \tilde x_{N, l}$.
	However, since two-marginal plan
	\[
	\bar\pi^{(N)} = \sum_j \bar\pi^{(N)}_j \delta(\tilde m_j^{(N-1)}, \tilde x_{N, j})
	\]
	in Algorithm~\ref{alg:greedy} has the sorting property by optimality, $\tilde m^{(N-1)}_k < \tilde m^{(N-1)}_l$ implies $\tilde x_{N, k}\leq \tilde x_{N, l}$.
\end{proof}

\section{Numerical Results}\label{sec:numerics}

We present a numerical comparison of different barycenter algorithms and, as applications, an interpolation between measures and textures, respectively.
To compute the two-marginal transport plans of the presented algorithms, we used the \texttt{emd} function of the Python-OT (POT 0.7.0) package \cite{flamary2021pot}, which is a wrapper of the network simplex solver from \cite{BPPH11networksimplex}, which, in turn, is based on an implementation in the LEMON C++ library.\footnote{\url{http://lemon.cs.elte.hu/pub/doc/latest-svn/index.html}}

\subsection{Ellipse Dataset}\label{sec:ellipses}

We compute Wasserstein barycenters with different algorithms for $\lambda\equiv 1/N$ and a data set of $N=10$ ellipses shown in Figure \ref{fig:ellipse_data}, which is originally from \cite{CD14fast} and has been commonly used as a benchmark example in the literature.
It is given as images of $60\times 60$ pixels.
We call the algorithms above ``Greedy'' and ``Reference'' below and compute $\tilde\nu=\push\tilde\pi$ of the resulting multi-marginal transport $\tilde\pi$ (without parallelization).\footnote{ \url{https://github.com/jvlindheim/mot}}
Further, we compute the barycenter using publicly available implementations for the methods \cite{JCG20debiased,GWXY19MAAIPM,LSPC2019frankwolfe}, called ``Debiased'', ``IBP'', ``Product'', ``MAAIPM'' and ``Frank--Wolfe'' below,\footnote{ \url{https://github.com/hichamjanati/debiased-ot-barycenters}}
the exact barycenter method from \cite{AB21fixedd} called ``Exact'' below,\footnote{\url{https://github.com/eboix/high_precision_barycenters}} and the method from \cite{LHXCJ20fastIBP} called ``FastIBP'' below.\footnote{ \url{https://github.com/tyDLin/FS-WBP}}
We also tried the BADMM\footnote{ \url{https://github.com/bobye/WBC_Matlab}} method from \cite{YWWL17badmm}, but since it did not converge properly, we do not consider it further.
While the fixed support methods receive the input as measures supported on $\{ 0, \dots, 59/60 \}\times \{ 0, \dots, 59/60 \}$, the free-support methods get the measures as a list of support positions and corresponding weights.
For all Sinkhorn methods, we used a parameter of $\varepsilon=0.002$ and otherwise chose the default parameters.
The order of the measures for the greedy and reference algorithm have been chosen according to the ordering in Figure~\ref{fig:ellipse_data}.
To compare the runtimes, we executed all codes on the same laptop with Intel i7-8550U CPU and 8GB memory running on Linux.
The Matlab codes were run in Matlab R2020a.
The runtimes of the Python codes are averages over several runs, as obtained by Python's \texttt{timeit} function.
The results are shown in Figure~\ref{fig:ellipse_barys} and Table~\ref{tab:ellipses}.

\renewcommand\curwidth{.9\textwidth}
\begin{figure}[htb]
	\centering
	\includegraphics[width=\curwidth]{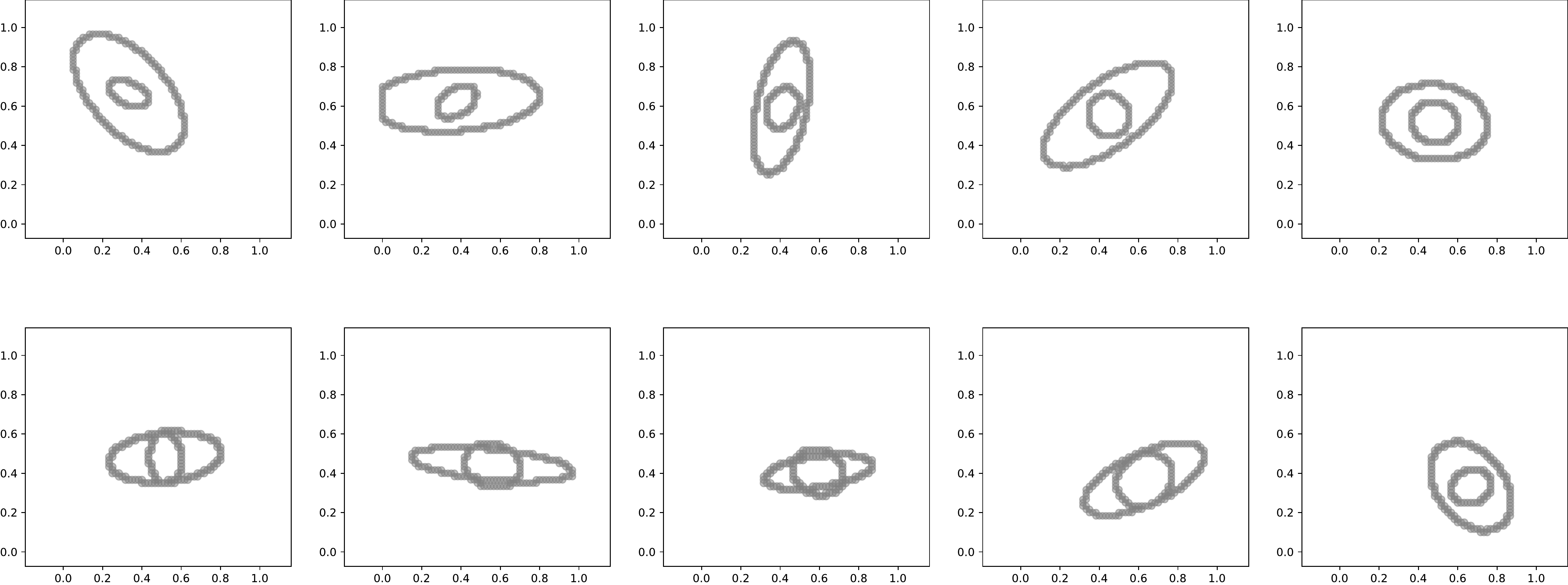}
	\caption{Data set of $10$ nested ellipses. }
	\label{fig:ellipse_data}
\end{figure}
\renewcommand\curwidth{.9\textwidth}
\begin{figure}[htb]
	\centering
	\includegraphics[width=\curwidth]{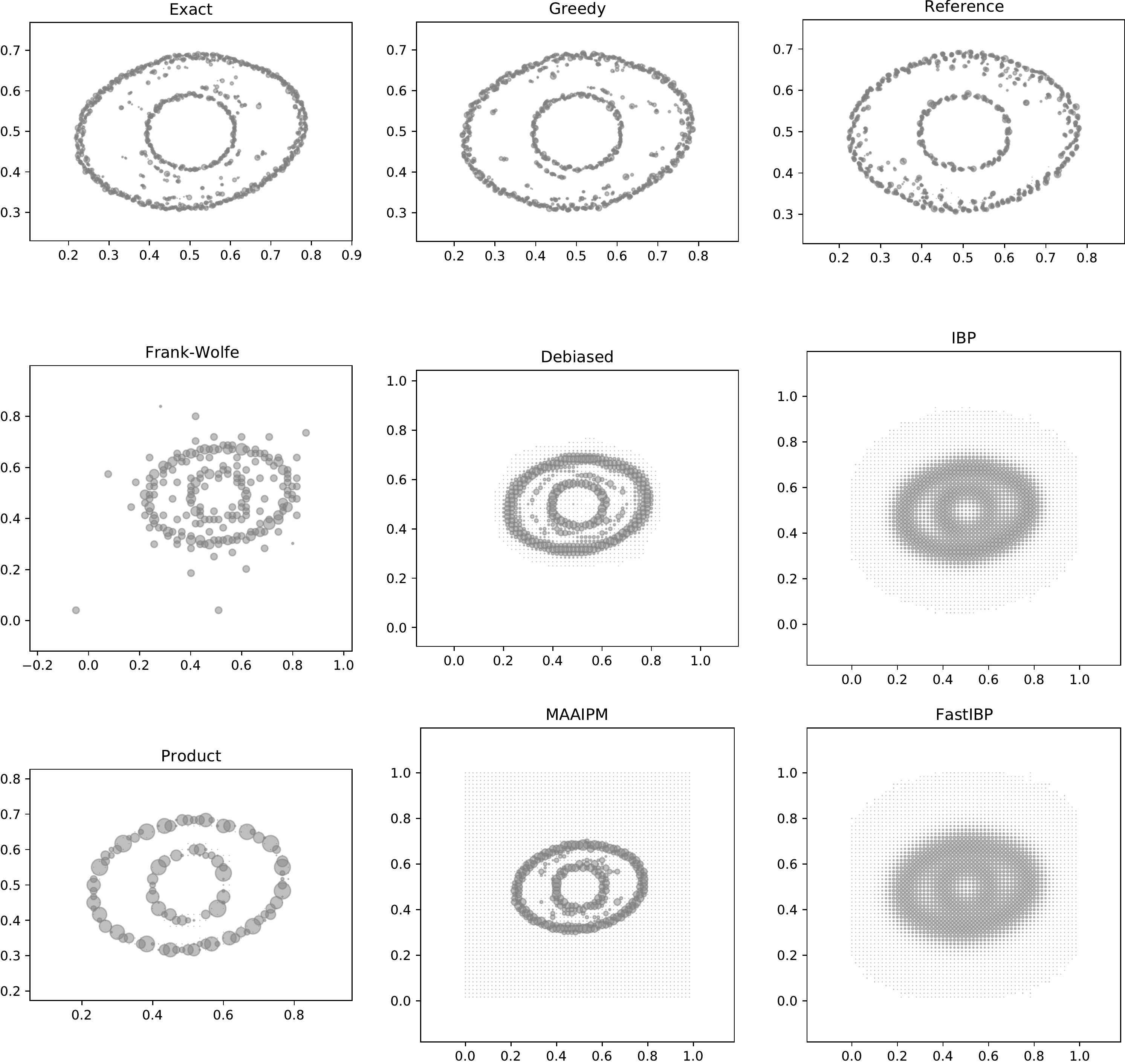}
	\caption{Barycenters for data set in Figure~\ref{fig:ellipse_data} computed by different methods. The weight of a support point is indicated by its area in the plot.}
	\label{fig:ellipse_barys}
\end{figure}

\begin{table}
	\centering
	\input{tables/results.tex}
	\caption{Numerical results for the ellipse barycenter problem. The error here is the absolute error defined as $1-\Psi(\tilde\nu)/\Psi(\hat\nu)$. The runtime is measured in seconds. The best scores of all approximative algorithms are highlighted in bold.}
	\label{tab:ellipses}
\end{table}

All methods achieve a relative error $\Psi(\tilde\nu)/\Psi(\hat\nu)$ well below $2$, which is a lot better than the worst case bounds shown above.
Whereas the greedy algorithm achieves the lowest error of all approximative algorithms, the reference algorithm achieves the lowest runtime.
Note that the support size of the $N-1$ two-marginal OT plans in the greedy algorithm is growing in each iteration, unlike as in the reference algorithm.
As another evaluation score, we computed the product of the error and the runtime, such that low scores in this metric indicate a good compromise between fast runtime and high precision.
Except for the exact method, which has no error but a very high runtime, the greedy and reference algorithms are the best with respect to this metric.
Notably, the FastIBP method is a lot slower than IBP whilst producing a more blurry result, which might indicate an implementation issue.
While the Frank--Wolfe method suffers from outliers, the support of most fixed-support methods is more extended than exact barycenter's support, since Sinkhorn-barycenters have dense support.
Altogether, the results of the proposed algorithms look promising.

\subsection{Multiple Different Sets of Weights}\label{sec:multiple_weights}

For this numerical application, we aim to interpolate between several given measures for multiple sets of weights $\lambda^k=(\lambda_1^k, \dots, \lambda_N^k)$, $\lambda^k\in \Delta_4$, $k=1, \dots, K$.
Encouraged by the precision results of the greedy algorithm in Section~\ref{sec:ellipses}, we apply it for this example, using two different approaches.
On the one hand, we can compute a $\tilde\pi^k$ and $\tilde\nu= (M_{\lambda^k})_\# \tilde\pi^k$ for each given $\lambda^k$, $k=1, \dots, K$ using Algorithm~\ref{alg:greedy}.
On the other hand, since this is the computational bottleneck, we can compute $\tilde\pi$ only once for e.g. $\lambda \equiv 1/N$ and then compute $\tilde\nu= (M_{\lambda^k})_\# \tilde\pi$ for all $k=1, \dots, K$.
We compare both approaches for a data set of four measures given as images of size $50\times 50$, for sets of weights that bilinearly interpolate between the four unit vectors.
The results are shown in Figures~\ref{fig:four_images_recompute}--\ref{fig:four_images_highlight}.
\begin{figure}[htb]
	\centering
	\includegraphics[width=.95\textwidth]{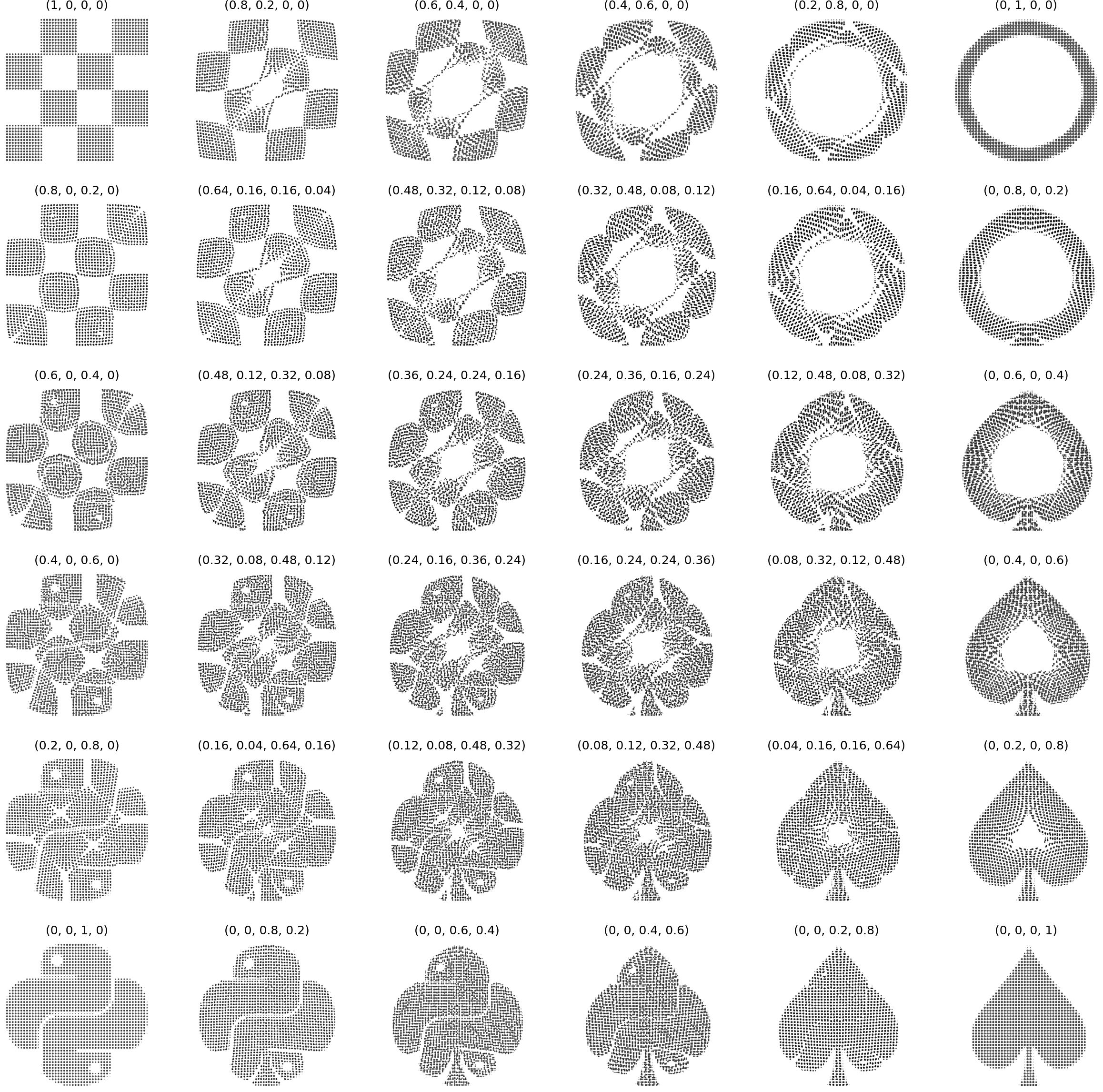}
	\caption{Barycenters $\tilde\nu= (M_{\lambda^k})_\# \tilde\pi^k$ of $\mu^1, \dots, \mu^4$ for different weight sets $\lambda^k$, where Algorithm~\ref{alg:greedy} has been run once for each weight set $\lambda^k$. The measures $\mu^1, \dots, \mu^4$ are shown in the four corners.}
	\label{fig:four_images_recompute}
\end{figure}
\begin{figure}[htb]
	\centering
	\includegraphics[width=.95\textwidth]{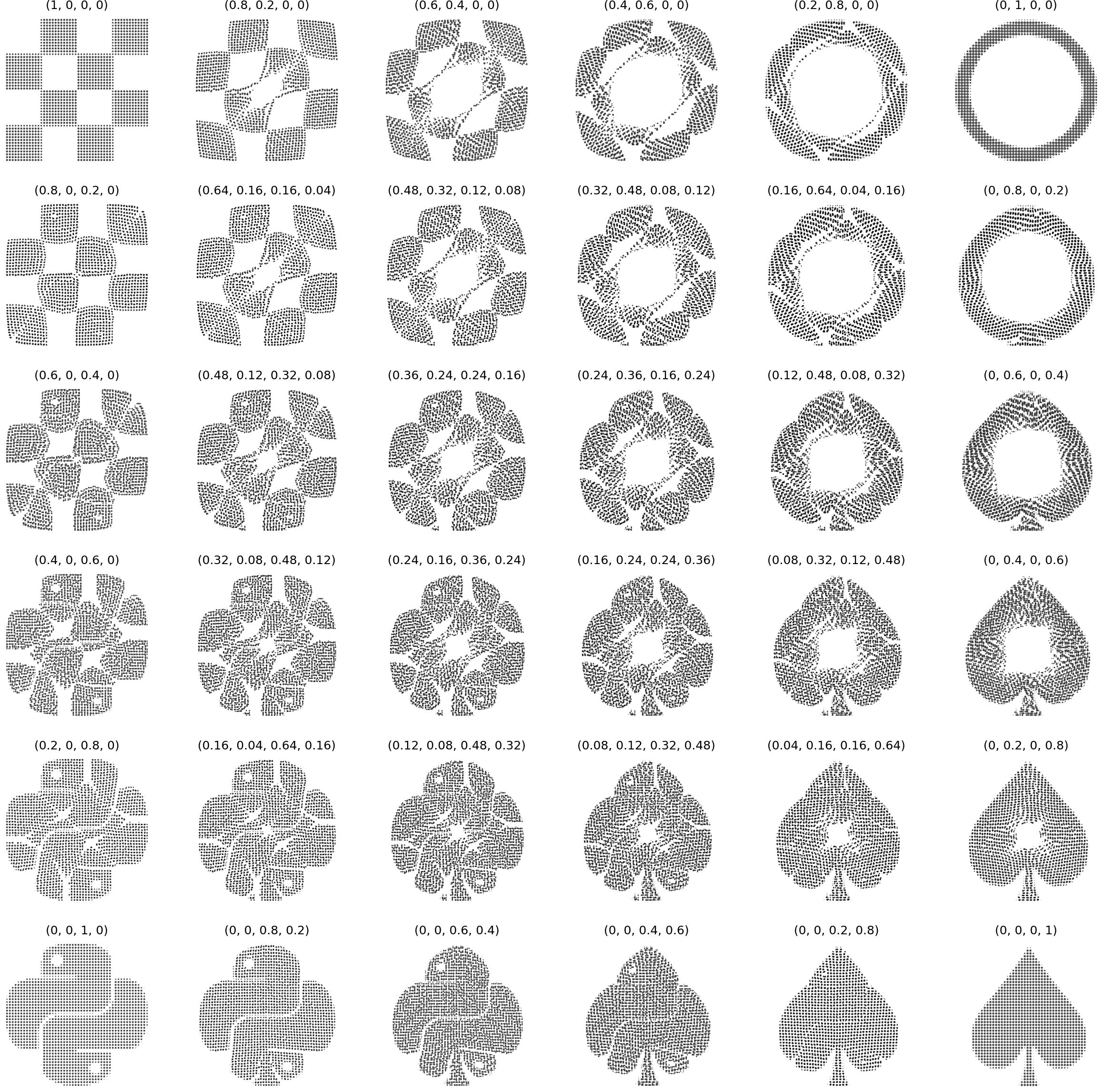}
	\caption{Computationally inexpensive barycenters $\tilde\nu= (M_{\lambda^k})_\# \tilde\pi$ of $\mu^1, \dots, \mu^4$ for different weight sets $\lambda^k$, where $\tilde\pi$ has been computed using Algorithm~\ref{alg:greedy} for $\lambda\equiv 1/N$.}
	\label{fig:four_images_greedy}
\end{figure}
\begin{figure}[htb]
\centering
\includegraphics[width=.95\textwidth]{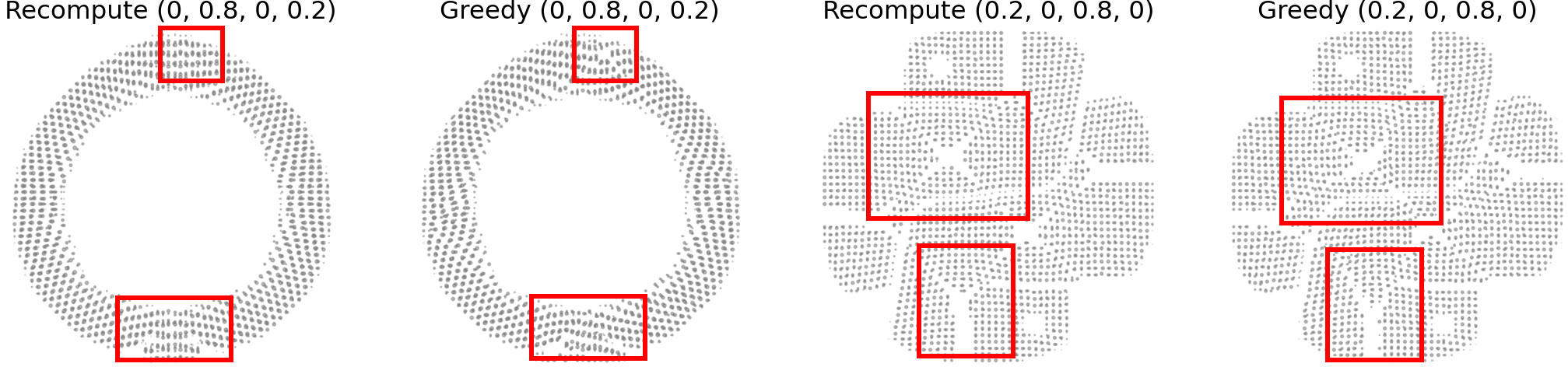}
\caption{Comparison of two interpolations as computed by the two different computational approaches, where ``recompute'' stands for an individual run of Algorithm~\ref{alg:greedy} for each weight set $\lambda^k$ and ``greedy'' stands for the approach using only one run of Algorithm~\ref{alg:greedy}.}
\label{fig:four_images_highlight}
\end{figure}

While the results from Figures~\ref{fig:four_images_recompute} and \ref{fig:four_images_greedy} look quite similar, we highlight small differences in the results in Figure~\ref{fig:four_images_highlight}.
In general, the interpolation from the recomputation approach looks a bit more smooth and has less artifacts.
The price to pay is of course a much higher computational cost, whereas for the second approach, only one run of Algorithm~\ref{alg:greedy} as a preparation is enough to be able to compute the other interpolations very fast.

\subsection{Texture Interpolation}

For another application, we lift the experiment of Section~\ref{sec:multiple_weights} from interpolation of measures in Euclidean space to interpolation of textures via the synthesis method from \cite{HLPR21gotex}, using their publicly available source code\footnote{\url{https://github.com/ahoudard/wgenpatex}}.
While the authors already interpolated between two different textures in that paper, requiring only the solution of a two-marginal optimal transport problem to obtain a barycenter, we can do this for multiple textures using Algorithm~\ref{alg:greedy}.
Briefly, the authors proposed to encode a texture as a collection of smaller patches $F_j$, where each, say, $4\times 4$-patch is encoded as a point $x_j\in \R^{16}$.
The texture is then modeled as a ``feature measure'' $\frac 1 M \sum_{j=1}^M\delta(x_j)\in \mathcal P(\R^{16})$,
such that this description is invariant under different positions of its patches within the image.
Finally, this is repeated for image patches at several scales $s$, obtaining a collection of measures $(\mu^s)$, $s=1, \dots, S$.
Synthesizing an image is done by optimizing an optimal transport loss between its feature measure and some reference measure (summing over $s$), as obtained, e.g., from a reference image.
Thus, the synthesized image tries to imitate the reference image in terms of its feature measures.
Here, we choose four texture images of size $256\times 256$ from the ``Describable Textures Dataset'' \cite{CMKMV14textures}.
We compute their feature measures $\mu^{1, s}, \dots, \mu^{4, s}$ for each scale.
Next, as in Section~\ref{sec:multiple_weights}, we compute their barycenters $\tilde\nu^{k, s}=(P_{\lambda^k})_\# \tilde\pi^{k, s}$ for all $k$ and $s$ and perform the image synthesis for each $k$ using the $\tilde\nu^{k, s}$ as feature measures to imitate.
The results are shown in Figure~\ref{fig:textures}.
Using this approach, one obtains visually pleasing interpolations between the four given textures.
\begin{figure}[htb]
	\centering
	\includegraphics[width=\curwidth]{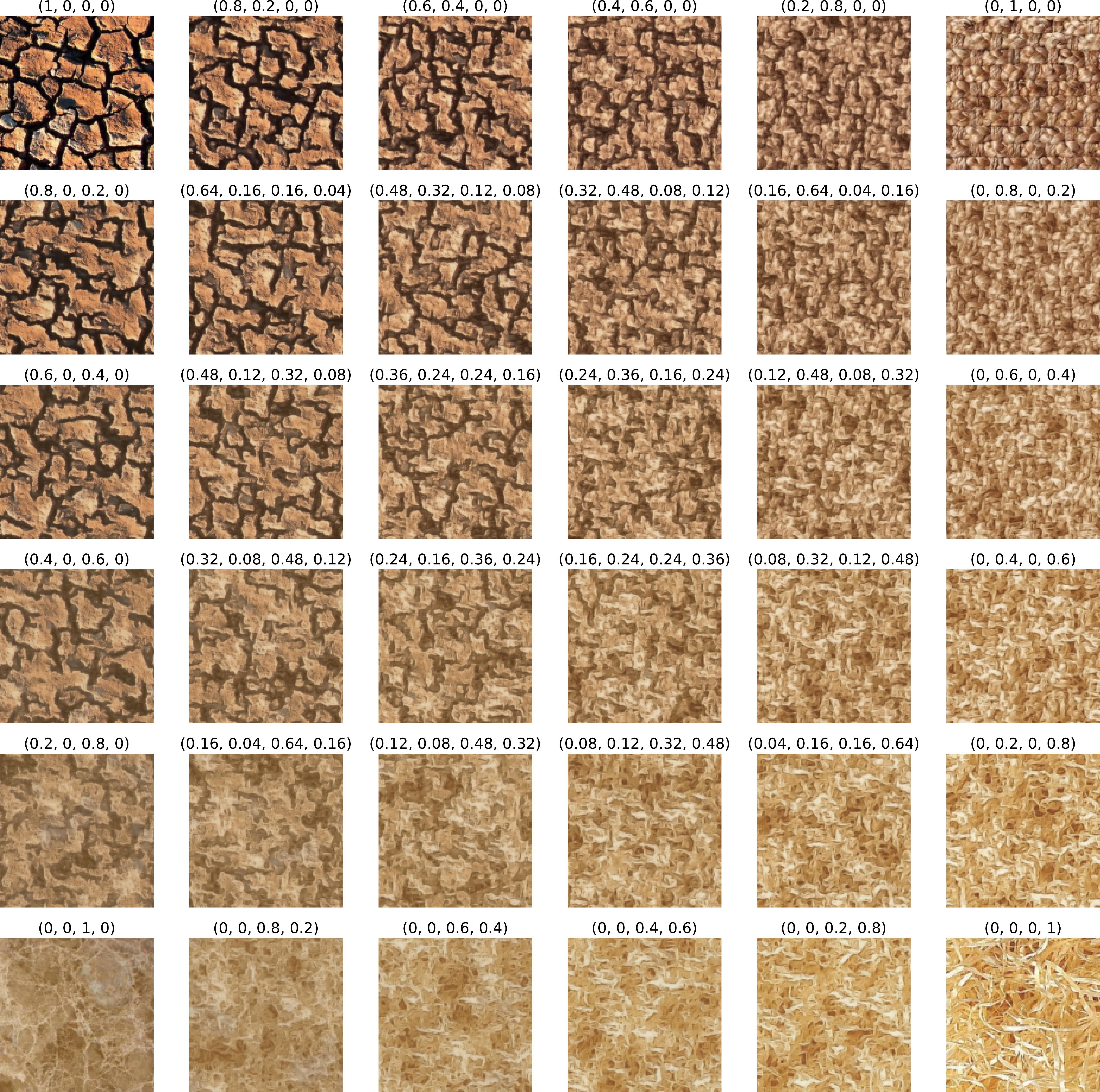}
	\caption{Interpolation of four different textures that are displayed in the four corners. The weight set for the barycenter computations performed for each image is shown above each synthesized image.}
	\label{fig:textures}
\end{figure}

\section{Discussion}\label{sec:discussion}
In this paper, after relating the MOT and barycenter problems with squared Euclidean cost, we introduced two approximative algorithms for them.
They are easy to implement, produce sparse solutions and are thus memory-efficient.
We analyzed both algorithms theoretically and validated their speed and precision in the numerical experiments.

In the future, we aim to close the gap between the upper and lower bound of the greedy algorithm.
Further, since \eqref{eq:non-splitting} is in general not met for suboptimal solutions $\tilde\pi$, as seen in Example~\ref{ex:ot_plans_cannot_be_read_off_of_mot}, this suggests another greedy iterative strategy that we would like to explore:
Given some $\tilde\pi$, for all $i$ cyclically, remove $\mu^i$ from $\tilde\pi$, i.e., form $(P_{1, \dots, i-1, i+1, \dots, N})_\#\tilde\pi$.
Now add $\mu^i$ back as in the greedy algorithm's last iteration, which produces a solution at least as good as before, and repeat until convergence.
For the reference algorithm, using reference measures other than $\mu^1$ is also possible, which could yield better theoretical or numerical results.
Finally, since two-marginal OT solvers like the simplex method handle arbitrary costs, we want to generalize the presented methods to MOT-problems with costs composed of other functions than $c$.

\section*{Acknowledgements}
Many thanks to Gabriele Steidl and Florian Beier for fruitful discussions.


\input{mot_algorithms.bbl}

\begin{appendix}
\section{Proofs}
We will present some examples below that show a lower bound to the relative error $\Phi(\tilde\pi)/\Phi(\hat\pi)$ of the presented algorithms, that is, examples where the algorithms perform badly.
These worst-case examples exploit the periodicity of the $1$-dimensional torus $\mathbb T$.
However, since we are interested in barycenters in the Euclidean space $\R^d$, we need the following lemma as a preparation.
It states that for ``enough periodic repetitions'' on the torus,
the examples also work for the corresponding $\R^2$-embeddings.

\begin{lemma}\label{lem:torus}
	Let $T:\mathbb T\to \R^2$, $\gamma\mapsto (\cos(\gamma), \sin(\gamma))$ 
	be the embedding of the torus into the two-dimensional Euclidean space.
	Then, for any $\alpha, \beta\in [0, \pi)$ with $\beta> \alpha$, it holds
	\[
	\lim_{s\to 0} \frac{\Vert T(s\alpha) - T(s\beta)\Vert}{s(\beta-\alpha)} = 1.
	\]
\end{lemma}

\begin{proof}
	We have 
	\begin{align*}
	\Vert T(s\alpha) - T(s\beta) \Vert
	&= 
	\Vert (\cos(s\beta ) - \cos(s\alpha), \sin(s\beta ) - \sin(s\alpha))\Vert \\
	&= 
	\sqrt{(\cos(s\beta ) - \cos(s\alpha))^2 + (\sin(s\beta ) - \sin(s\alpha))^2} \\
	&=
	\sqrt{2-2\cos(s\beta)\cos(s\alpha) - 2\sin(s\beta)\sin(s\alpha)}\\
	&= \sqrt{2-2\cos(s(\beta-\alpha))} 
	= 
	2\sin(\frac 1 2 s(\beta-\alpha)),
	\end{align*}
	where we have used the identity $1-\cos(x) = 2\sin^2(x/2)$ for the last equality.
	Since $\sin(x)/x\to 1$ for $x\to 0$, we finally get
	\begin{align*}
	\lim_{s\to 0} \frac{\Vert T(s\alpha) - T(s\beta)\Vert}{s(\beta-\alpha)}
	= \lim_{s\to 0}\frac{2\sin(\frac 1 2 s(\beta-\alpha))}{s(\beta-\alpha)}
	= \lim_{s\to 0}\frac{\sin(\frac 1 2 s(\beta-\alpha))}{\frac 1 2 s(\beta-\alpha)}
	= 1.
	\end{align*}
\end{proof}

\subsection{Reference Algorithm -- Lower Bound}
After the previous preparation, we prove Theorem~\ref{thm:ref-lower}.
\begin{proof}
	Set $\lambda\equiv 1/N$ and choose 
	\begin{align*}
	& x^1_1 = 0, \\
	& x^2_1 = x^4_1 = \dots = \frac 1 {1+\tilde\varepsilon}\cdot \frac{\pi}{M}, \\
	& x^3_1 = x^5_1 = \dots = -\frac 1 {1+\tilde\varepsilon}\cdot \frac{\pi}{M}, \\
	& x^i_j = x^i_1 + (j-1)\cdot \frac{2\pi}{M} \text{ for all } i=1, \dots, N, j =2,\dots, M
	\end{align*}
	for some $0<\tilde\varepsilon$ small enough, $1<M\in \mathbb N$, with slight abuse of notation (denoting here by $\pi$ the area of the unit circle).
	Finally, choose
	\[
	\mu_i= \frac 1 M \sum_{j=1}^M \delta(x^i_j) \text{ for all } i = 1, \dots, N.
	\]
	For a sketch of this example, see Figure~\ref{fig:worst_case_ref}.
	\begin{figure}[htb]
		\renewcommand\curwidth{.7\textwidth}
		\centering
		\includegraphics[width=\curwidth]{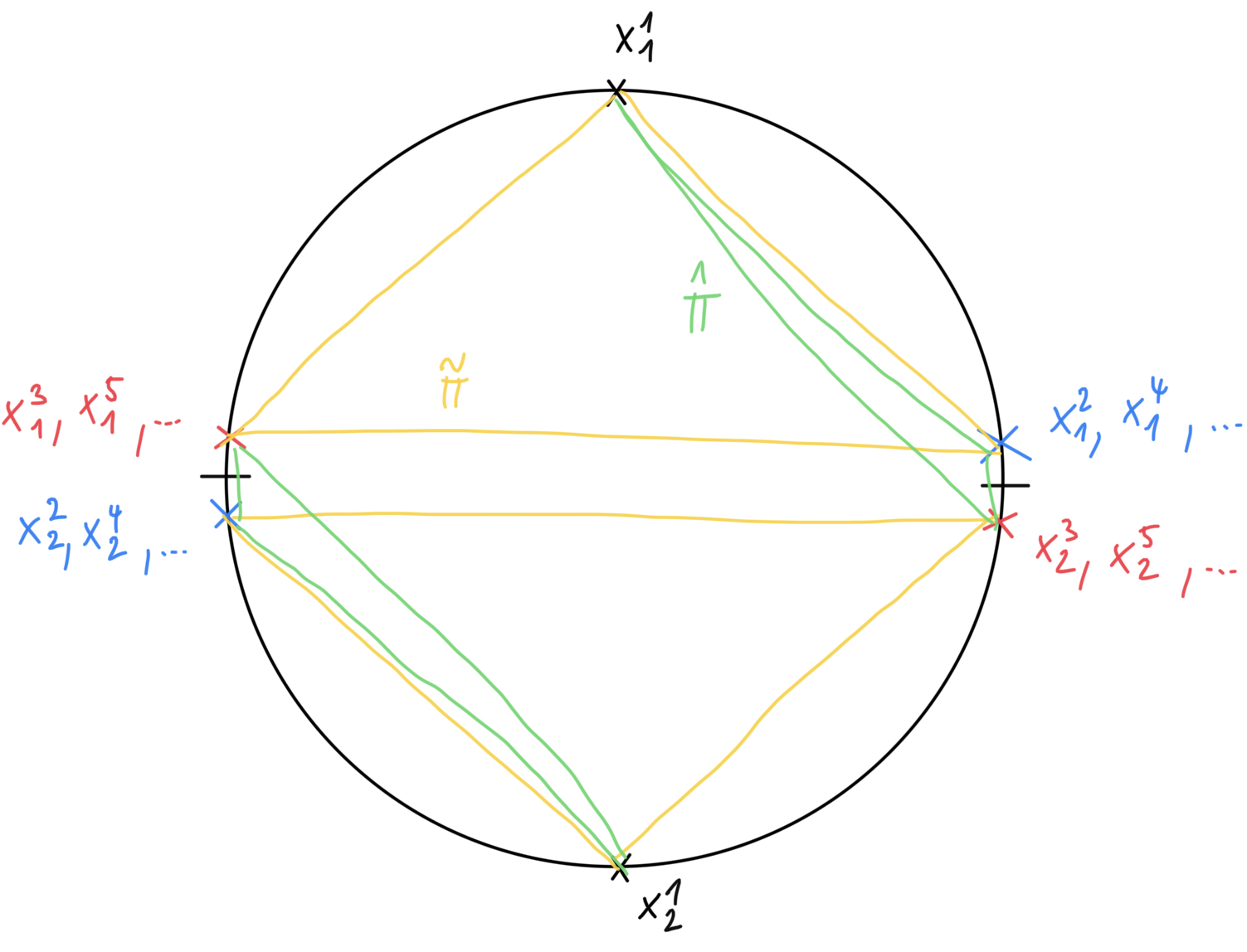}
		\caption{Sketch of a worst-case example for the reference algorithm for the case $M=2$, i.e., for each transport plan, there is $M=2$ tuples. The green lines between points indicate the tuples of $\hat\pi$, whereas the yellow lines indicate the tuples chosen by Algorithm~\ref{eq:ref_alg}.}
		\label{fig:worst_case_ref}
	\end{figure}
	
	For the costs, we use the squared distance on the torus.
	First, assume $N$ to be odd.
	From Figure~\ref{fig:worst_case_ref}, it becomes clear what the multi-marginal plan $\tilde\pi$ chosen by the reference heuristic and the optimal plan $\hat\pi$ are:
	The optimal two-marginal couplings chosen by the reference heuristic are just coupling the points of $\mu_1$ with their nearest neighbors, so that we get
	\[
	\tilde\pi = \frac 1 M \sum_{j=1}^M \delta(x^1_j, \dots, x^N_j).
	\]
	On the other hand, setting $x^i_{M+1}=x^i_1$ for all $i=1, \dots, N$, consider the plan $\hat\pi$ as sketched in Figure~\ref{fig:worst_case_ref} given as
	\begin{equation}\label{eq:worst_case_ref_opt}
		\hat\pi = \frac 1 M \sum_{j=1}^M \delta(x^1_j, x^2_j, x^3_{j+1}, x^4_j.  \dots, x^N_{j+1})
	\end{equation}
	Next, we compute the costs $\Phi(\tilde\pi)$ and $\Phi(\hat\pi)$.
	By definition of $\Phi$, we have
	\begin{align*}
	\Phi(\tilde\pi) = \sum_{j=1}^M \frac 1 M \sum_{s<t} \frac 1 {N^2} \vert x^s_j - x^t_j\vert^2.
	\end{align*}
	Note that $\vert x^s_j - x^t_j\vert^2 = \vert x^s_1 - x^t_1\vert^2$ for all $j=2, \dots, M$ and all $s, t = 1, \dots, N$.
	Further, $\vert x^s_1 - x^t_1\vert^2 = 0$ whenever $s, t \geq 2$ and both are even or both are odd.
	Using this, and taking $N^2$ to the other side, we get
	\begin{align*}
	N^2 \Phi(\tilde\pi)
	&= \sum_{s<t}^N  \vert x^s_1 - x^t_1\vert^2
	= \sum_{i=2}^N \vert x^1_1 - x^i_1\vert^2 + \sum_{\substack{2\leq s< t \\ (s \,\mathrm{mod}\, 2) \neq (t\,\mathrm{mod}\, 2)}} \vert x^s_1 - x^t_1\vert^2 \\
	&= (N-1) \Big( \frac{\pi}{M} \Big)^2 \frac{1}{(1+\tilde\varepsilon)^2}
	+ \Big(\frac{N-1}{2}\Big)^2 \Big( \frac{\pi}{M} \Big)^2 \frac{2^2}{(1+\tilde\varepsilon)^2}
	= \Big( \frac{\pi}{M} \Big)^2 \frac{N(N-1)}{(1+\tilde\varepsilon)^2}.
	\end{align*}
	Similarly $\vert x^s_1-x^t_2\vert^2 = \vert x^s_j-x^t_{j+1}\vert$ for all $j=2, \dots, M$ and $s,t=1, \dots N$.
	Thus, from the coupling \eqref{eq:worst_case_ref_opt}, we get
	\begin{align*}
	N^2\Phi(\hat\pi)
	&\leq \sum_{2\leq i \text{ even}} \vert x^1_1 - x^i_1\vert^2
	+ \sum_{2\leq i \text{ odd}} \vert x^1_1 - x^i_1\vert^2
	+ \sum_{\substack{2\leq s< t \\ (s \,\mathrm{mod}\, 2) \neq (t\,\mathrm{mod}\, 2)}} \vert x^s_1 - x^t_2\vert^2 \\
	&= \Big( \frac{\pi}{M} \Big)^2\Big(\frac{N-1}{2}\Big(\Big( 1-\frac{\tilde\varepsilon}{1+\tilde\varepsilon} \Big)^2 + \Big(1+\frac{\tilde\varepsilon}{1+\tilde\varepsilon} \Big)^2 \Big)
	+  \Big(\frac{N-1}{2}\Big)^2\Big( \frac{2\tilde\varepsilon}{1+\tilde\varepsilon} \Big)^2
	\Big) \\
	&\leq \Big( \frac{\pi}{M} \Big)^2\Big((N-1)\Big( \frac{1+2\tilde\varepsilon}{1+\tilde\varepsilon} \Big)^2 + (N-1)^2 \frac{\tilde\varepsilon^2}{(1+\tilde\varepsilon)^2}\Big) \\
	&= \Big( \frac{\pi}{M} \Big)^2\frac{N-1}{(1+\tilde\varepsilon)^2}\Big( (1+2\tilde\varepsilon)^2 + \tilde\varepsilon^2(N-1) \Big).
	\end{align*}
	Thus we get
	\begin{align*}
	\frac{\Phi(\tilde\pi)}{\Phi(\hat\pi)}
	\geq \frac{N}{(1+2\tilde\varepsilon)^2 + \tilde\varepsilon^2(N-1)} \overset{\tilde\varepsilon\to 0}{\longrightarrow} N.
	\end{align*}
	Using the same example for even $N$, similarly we obtain
	\begin{align*}
		N^2 \Phi(\tilde\pi)
		= (N-1) \Big( \frac{\pi}{M} \Big)^2 \frac{1}{(1+\tilde\varepsilon)^2}
		+ \frac N 2\Big(\frac N 2 - 1\Big) \Big( \frac{\pi}{M} \Big)^2 \frac{2^2}{(1+\tilde\varepsilon)^2}
		= \Big( \frac{\pi}{M} \Big)^2 \frac{N(N-1)-1}{(1+\tilde\varepsilon)^2}
	\end{align*}
	and
	\begin{align*}
	N^2\Phi(\hat\pi)
	&= \Big( \frac{\pi}{M} \Big)^2\Big(\frac{N}{2}\Big( 1-\frac{\tilde\varepsilon}{1+\tilde\varepsilon} \Big)^2 + \Big(\frac N 2 - 1\Big)\Big(1+\frac{\tilde\varepsilon}{1+\tilde\varepsilon} \Big)^2
	+ \frac N 2\Big(\frac{N}{2}-1\Big)\Big( \frac{2\tilde\varepsilon}{1+\tilde\varepsilon} \Big)^2
	\Big) \\
	&\leq \Big( \frac{\pi}{M} \Big)^2\frac{1}{(1+\tilde\varepsilon)^2}\Big( (N-1)(1+2\tilde\varepsilon)^2 + N(N-2)\tilde\varepsilon^2 \Big),
	\end{align*}
	such that
	\begin{align*}
		\frac{\Phi(\tilde\pi)}{\Phi(\hat\pi)}
		\geq \frac{N(N-1)-1}{(N-1)(1+2\tilde\varepsilon)^2 - N(N-2)\tilde\varepsilon^2} \overset{\tilde\varepsilon\to 0}{\longrightarrow} N-\frac 1 {N-1}.
	\end{align*}
	
	Finally, we would like to obtain the statement for Euclidean distances as well.
	To this end, we embed the points from the torus into $\R^2$ using $T$ as defined in Lemma~\ref{lem:torus}.
	When calculating $\Phi(\tilde\pi)$ and $\Phi(\hat\pi)$ for the embedded example, the Euclidean distances can be rewritten as
	\[
	\Vert T(x^s_j)-T(x^t_l)\Vert^2 = \frac{\Vert T(x^s_j)-T(x^t_l)\Vert^2}{\vert x^s_j - x^t_l\vert^2} \vert x^s_j - x^t_l\vert^2,
	\]
	where the fraction goes to $1$ for $M\to\infty$ by Lemma~\ref{lem:torus}.
	Thus, for $M$ large enough, we get arbitrarily close to the result using the torus distances.
	Altogether, for $\tilde\varepsilon$ small enough and $M$ large enough, we obtain the desired result.
	
	The same example shows that \eqref{eq:ref_rand_bound} is asymptotically tight for growing $N$:
	Choose a sequence $(\varepsilon_N)$ with $\varepsilon_N / N \to 0$ for $N\to \infty$, and choose for each odd $N$ the example above with $\tilde\varepsilon = \varepsilon_N$.
	Thus, for each $N$, we get plans $\tilde\pi(N)$ and $\hat\pi(N)$.
	Since we choose $\mu_1$ with probability $1/N$, we get
	\begin{align*}
	\frac{\mathbb E [\Phi(\tilde \pi(N))]}{\Phi(\hat\pi(N))}
	\geq \frac 1 N (N-\varepsilon_N) + \frac{N-1}{N} \cdot 1
	\overset{N\to\infty}{\longrightarrow} 2.
	\end{align*}
	This concludes the proof.	
\end{proof}


\subsection{Greedy Algorithm -- Upper Bound}
Next, we prove Theorem~\ref{thm:greedy-upper} on the upper bound of the greedy algorithm's relative error $\Phi(\tilde\pi)/\Phi(\hat\pi)$.
\begin{proof}
	The proof is given by induction over $N$.
	For $N=2$, we have $\Phi(\tilde\pi)= \Phi(\hat\pi)$, so the statement is clear since the right-hand side of \eqref{eq:greedy_upper} is $1$ in this case.
	In the following, assume $N>2$.
	Let 
	\[
	\tilde \pi = \sum_{j=1}^{\tilde M} \tilde \pi_j \delta(\tilde x_{1, j}, \dots, \tilde x_{N, j}).
	\]
	In this proof, write $\bar \lambda_i = \bar \lambda_{i, N-1}$.
	By construction of the algorithm and the marginal constraints on $\tilde\pi$, the plan
	\[
	(P_{1, \dots, N-1})_\# \tilde\pi = \tilde \pi^{(N-1)} = \sum_{j=1}^{\tilde M_{r-1}} \tilde \pi_j \delta(\tilde x_{1, j}, \dots, \tilde x_{N-1, j})
	\]
	is the plan obtained by the algorithm before the last iteration.
	We denote by
	\[
	\tilde\nu^{(N-1)}
	= (M_{\bar\lambda})_\# \tilde\pi
	= \sum_{j=1}^{\tilde M_{r-1}} \tilde\pi_j \delta(\tilde m_j^{(N-1)}), \quad
	\tilde m_j^{(N-1)} = \sum_{i=1}^{N-1} \bar \lambda_i \tilde x_{i, j}
	\]
	the corresponding barycenter.
	Further, we make the same definitions for an optimal plan $\hat\pi$, only exchanging all tildes for hats above.
	Recall the notation $\mu^i = \sum_{j=1}^{n_i} \mu_j^i x_j^i$ from \eqref{eq:mu_def}.
	
	We construct the following, not necessarily optimal, couplings between $\tilde \nu^{(N-1)}$ and $\mu_N$, one for every $i=1, \dots, N-1$, as follows.
	Fix $i$ and write $\mu_N = \sum_l \hat \pi_l \delta(\hat x_{N, l})$, which is possible by the marginal constraints on $\hat\pi$.
	Then define
	\[
	\pi^i
	= \sum_{k, l} \pi_{k, l}^i \delta(\tilde m_k^{(N-1)}, \hat x_{N, l}), \quad
	\pi_{k, l}^i = \delta_{\tilde x_{i, k} = \hat x_{i, l} = x_j^i} \frac {\tilde \pi_{k} \hat \pi_{l}}{\mu^i_j},
	\]
	where $\delta$ in the last definition denotes the Kronecker delta.
	Written differently,
	\[
	\pi^i
	= \sum_{j=1}^{n_i} \sum_{\substack{k \\ \tilde x_{i, k} = x_j^i}}\sum_{\substack{l \\ \hat x_{i, l} = x_j^i}} \frac {\tilde \pi_{k} \hat \pi_{l}}{\mu^i_j}\delta(\tilde m_{k}^{(N-1)}, \hat x_{N, l}).
	\]
	That is, we couple the barycenter support point $\tilde m_{j}^{(N-1)}$ of the algorithm with all support points $\hat x_{N, l}$ of $\mu_N$
	that have an intersection in the $i$-th coordinate in their corresponding tuples of the multi-marginal plans $\hat\pi$ resp. $\tilde\pi$ they correspond to.
	We show that the marginal constraints $\pi^i\in \Pi(\tilde\nu^{(N-1)}, \mu^N)$ are met: Suppose $k$ is fixed, i.e., $\tilde x_{i, k} = x_j^i$ is fixed for some $j$, then
	\[
	\sum_l \pi^i_{k, l}
	= \sum_{\substack{l \\ \hat x_{i, l} = x_j^i}} \frac {\tilde \pi_{k} \hat \pi_{l}}{\mu^i_j}
	= \frac{\tilde \pi_{k}}{\mu^i_j}\sum_{\substack{l \\ \hat x_{i, l} = x_j^i}} \hat \pi_{l} = \tilde \pi_{k}.
	\]
	If $l$ is fix, i.e., $\hat x_{N, l}$ is fix, then $(\hat x_{1, l}, \dots, \hat x_{N, l})$ is fix and hence $\hat x_{i, l} = x_j^i$ is fix for some $j$.
	Thus
	\[
	\sum_k \pi^i_{k, l}
	= \sum_{\substack{k \\ \tilde x_{i, k} = x_j^i}} \frac {\tilde \pi_{k} \hat \pi_{l}}{\mu^i_j}
	= \frac{\hat \pi_{l}}{\mu^i_j}\sum_{\substack{k \\ \tilde x_{i, k} = x_j^i}} \tilde \pi_{k} = \hat \pi_{l}.
	\]
	Since the algorithm chooses the optimal coupling between $\tilde\nu^{(N-1)}$ and $\mu_N$ by construction, it holds
	\begin{align*}
	& \sum_{j=1}^M \tilde \pi_j \Vert \tilde m_j^{(N-1)} - \tilde x_{N, j} \Vert^2
	= \sum_{i=1}^{N-1} \bar\lambda_i \sum_{j=1}^M \tilde \pi_j \Vert \tilde m_j^{(N-1)} - \tilde x_{N, j} \Vert^2
	\leq \sum_{i=1}^{N-1} \bar\lambda_i \langle c, \pi^i \rangle \\
	=\;& \sum_{i=1}^{N-1} \bar\lambda_i \sum_{j=1}^{n_i} \sum_{\substack{k \\ \tilde x_{i, k} = x_j^i}}\sum_{\substack{l \\ \hat x_{i, l} = x_j^i}} \frac {\tilde \pi_{k} \hat \pi_{l}}{\mu^i_j} \Vert \tilde m^{(N-1)}_{k} - \hat x_{N, l}\Vert^2 \\
	\leq\;& 2\sum_{i=1}^{N-1}\bar\lambda_i  \sum_{j=1}^{n_i} \sum_{\substack{k \\ \tilde x_{i, k} = x_j^i}}\sum_{\substack{l \\ \hat x_{i, l} = x_j^i}} \frac {\tilde \pi_{k} \hat \pi_{l}}{\mu^i_j} ( \Vert \tilde m^{(N-1)}_{k} - x_j^i \Vert^2 + \Vert x_j^i- \hat x_{N, l}\Vert^2) \\
	=\;& 2 \sum_{i=1}^{N-1}\bar\lambda_i \sum_{j=1}^{n_i}\sum_{\substack{k \\ \tilde x_{i, k} = x_j^i}} \tilde \pi_{k}\Vert \tilde m^{(N-1)}_{k} - x_j^i \Vert^2 + 2 \sum_{i=1}^{N-1}\bar\lambda_i \sum_{j=1}^{n_i}\sum_{\substack{l \\ \hat x_{i, l} = x_j^i}} \hat \pi_{l}\Vert x_j^i- \hat x_{N, l}\Vert^2 \\
	=\;& 2 \sum_{i=1}^{N-1}\bar\lambda_i \sum_{j=1}^M \tilde \pi_{j}\Vert \tilde m^{(N-1)}_{j} - \tilde x_{i, j} \Vert^2 + 2 \sum_{i=1}^{N-1}\bar\lambda_i \sum_{j=1}^M \hat \pi_{j}\Vert \hat x_{i, j}- \hat x_{N, l}\Vert^2
	\end{align*}
	and using Lemma~\ref{lem:jensen_eq}, this equals
	\begin{align*}
	2 \sum_{ij} \bar \lambda_i \tilde \pi_j \Vert \tilde m^{(N-1)}_j - \tilde x_{i, j} \Vert^2 + 2\sum_{ij}\bar\lambda_i \hat \pi_j \Vert \hat x_{i, j} - \hat m_j^{(N-1)} \Vert^2 + 2\sum_{j=1}^M \hat \pi_j \Vert \hat m_j^{(N-1)} - \hat x_{N, j}\Vert^2.
	\end{align*}
	Next, we decompose the cost $\Phi(\tilde\pi)$ in terms of an induction hypothesis part and this new part.
	First note that, since $\tilde m_j = (1-\lambda_N) \tilde m_j^{(N-1)} + \lambda_N \tilde x_{N, j}$, we have
	\[
	\Vert \tilde m_j^{(N-1)} - \tilde m_j\Vert^2
	= \lambda_N^2 \Vert \tilde m_j^{(N-1)} - \tilde x_{N, j}\Vert^2.
	\]
	Then we compute, using \eqref{eq:Phi_2} and Lemma~\ref{lem:jensen_eq},
	\begin{align}
	\Phi(\tilde\pi)
	&= \sumitoNlambdai \sum_{j=1}^M\tilde \pi_j \Vert \tilde x_{i, j} - \tilde m_j\Vert^2
	= \sumitoNlambdai \sum_{j=1}^M\tilde \pi_j (\Vert \tilde x_{i, j} - \tilde m_j^{(N-1)}\Vert^2 - \Vert \tilde m_j^{(N-1)} - \tilde m_j\Vert^2) \nonumber \\
	&= -\lambda_N^2 \sum_{j=1}^M \tilde \pi_j \Vert \tilde m_j^{(N-1)} - \tilde x_{N, j}\Vert^2 + \lambda_N \sum_{j=1}^M \tilde \pi_j \Vert \tilde x_{N, j} - \tilde m_j^{(N-1)} \Vert^2 \nonumber\\
	& \quad + (1-\lambda_N)\sum_{i=1}^{N-1} \bar \lambda_i \sum_{j=1}^M \tilde \pi_j \Vert \tilde x_{i, j} - \tilde m_j^{(N-1)} \Vert^2 \nonumber\\
	&= (1-\lambda_N) \Big( \lambda_N \sum_{j=1}^M \tilde \pi_j \Vert \tilde x_{N, j} - \tilde m_j^{(N-1)} \Vert^2 + \sum_{i=1}^{N-1} \bar \lambda_i \sum_{j=1}^M \tilde \pi_j \Vert \tilde x_{i, j} - \tilde m_j^{(N-1)} \Vert^2\Big) \label{eq:inductive_split_costs}.
	\end{align}
	Note that the same calculation can be made for $\Phi(\hat\pi)$, again swapping all tildes for hats.
	Thus we get
	\begin{align*}
	&\frac{\Phi(\tilde\pi)}{\Phi(\hat\pi)}
	= \frac{\lambda_N \sum_j \tilde \pi_j \Vert \tilde x_{N, j} - \tilde m_j^{(N-1)} \Vert^2 + \sum_{i=1}^{N-1} \bar \lambda_i \sum_j \tilde \pi_j \Vert \tilde x_{i, j} - \tilde m_j^{(N-1)} \Vert^2}
	{\lambda_N \sum_j \hat \pi_j \Vert \hat x_{N, j} - \hat m_j^{(N-1)} \Vert^2 + \sum_{i=1}^{N-1} \bar \lambda_i \sum_j \hat \pi_j \Vert \hat x_{i, j} - \hat m_j^{(N-1)} \Vert^2} \\
	&\leq \resizebox{0.96\hsize}{!}{$\frac{2\lambda_N (\sum_j \hat \pi_j \Vert \hat m_j^{(N-1)} - \hat x_{N, j}\Vert^2
			+ \sum_{ij}\bar\lambda_i \hat \pi_j \Vert \hat x_{i, j} - \hat m_j^{(N-1)} \Vert^2)
			+ (1+2\lambda_N)\sum_{i, j} \bar \lambda_i \tilde \pi_j \Vert \tilde x_{i, j} - \tilde m_j^{(N-1)} \Vert^2}
		{\lambda_N \sum_j \hat \pi_j \Vert \hat x_{N, j} - \hat m_j^{(N-1)} \Vert^2 + \sum_{i, j} \bar \lambda_i \hat \pi_j \Vert \hat x_{i, j} - \hat m_j^{(N-1)} \Vert^2}$} \\
	&\leq \frac{2\lambda_N \sum_j \hat \pi_j \Vert \hat m_j^{(N-1)} - \hat x_{N, j}\Vert^2}{\lambda_N \sum_j \hat \pi_j \Vert \hat m_j^{(N-1)} - \hat x_{N, j}\Vert^2}
	+ \frac{2\lambda_N \sum_{i, j} \bar \lambda_i \hat \pi_j \Vert \hat x_{i, j} - \hat m_j^{(N-1)} \Vert^2}{\sum_{i, j} \bar \lambda_i \hat \pi_j \Vert \hat x_{i, j} - \hat m_j^{(N-1)} \Vert^2} \\
	& \quad + \frac{(1+2\lambda_N) \sum_{i, j} \bar \lambda_i \tilde \pi_j \Vert \tilde x_{i, j} - \tilde m_j^{(N-1)} \Vert^2}{\sum_{i, j} \bar \lambda_i \hat \pi_j \Vert \hat x_{i, j} - \hat m_j^{(N-1)} \Vert^2} \\
	&= 2 + 2\lambda_N + \frac{(1+2\lambda_N) \sum_{i, j} \bar \lambda_i \tilde \pi_j \Vert \tilde x_{i, j} - \tilde m_j^{(N-1)} \Vert^2}{\sum_{i, j} \bar \lambda_i \hat \pi_j \Vert \hat x_{i, j} - \hat m_j^{(N-1)} \Vert^2}.
	\end{align*}
	By induction hypothesis, this is at most
	\begin{align*}
	2+2\lambda_N + (1+2\lambda_N)\Big(\frac 2 3 (N-1)^2-\frac 5 3\Big),
	\end{align*}
	and because of $\lambda_N \leq \lambda_1, \dots, \lambda_{N-1}$, i.p. $\lambda_N\leq \frac 1 N$, this is at most
	\begin{align*}
	2+\frac 2 N + \Big(1+\frac 2 N \Big)\Big(\frac 2 3 (N-1)^2-\frac 5 3\Big).
	\end{align*}
	Finally, by elementary computations, this equals
	\[
	\frac 1 3 (2N^2 - 5),
	\]
	which concludes the  proof.
\end{proof}

\subsection{Randomized Greedy Algorithm -- Upper Bound}

Next, we prove Theorem~\ref{thm:greedy-upper-rand} on the upper bound of the greedy algorithm's expected relative error for $\lambda\equiv 1/N$ for a permutation of the input measures that is chosen uniformly at random.
The proof combines arguments from the proofs of Theorems~\ref{thm:ref-upper} and \ref{thm:greedy-upper} on the upper bounds of Algorithms~\ref{alg:ref} and \ref{alg:greedy}.
\begin{proof}
	First, we present the core idea, preparing the induction later on.
	Given some fixed
	\[
	\tilde \pi = \sum_{j =1}^{\tilde M} \tilde \pi_j \delta(\tilde x_{1, j}, \dots, \tilde x_{N, j}),
	\]
	we can write
	\[
	\mu^i = \sum_{j=1}^{\tilde M} \tilde \pi_j \delta(\tilde x_{i, j}) \quad \text{for all}\quad  i=1,\dots, N-1
	\quad\text{and}\quad
	\mu^N=\sum_{k=1}^{n_N} \mu_k^N\delta(x_k^N).
	\]
	Let
	\[
	\pi^{i, N} = \sum_{j, k} \pi^{i, N}_{j, k}\delta(\tilde x_{i, j}, x_k^N) \in \argmin_{\pi\in\Pi(\mu^i, \mu^N)}\langle c, \pi\rangle,
	\]
	such that
	\[
	\sum_{k=1}^{n_N} \pi^{i, N}_{j, k} = \tilde\pi_j \quad
	\text{and}\quad 
	\sum_{j=1}^{\tilde M}\pi^{i, N}_{j, k} = \mu_k^N.
	\]
	Then we set
	\[
	\pi^i \coloneqq \sum_{j, k} \pi^{i, N}_{j, k}\delta(\tilde m^{(N-1)}_j, x_k^N) \in \Pi(\tilde\nu^{(N-1)}, \mu^N).
	\]
	We use $\bar\lambda=(\bar\lambda_{1, N-1}, \dots, \bar\lambda_{N-1, N-1})$ as shorthand notation as above.
	With a similar computation as in the proof of Theorem~\ref{thm:greedy-upper}, we obtain
	\begin{align*}
	&\sum_{j=1}^M \tilde\pi_j \Vert \tilde m^{(N-1)}_j - \tilde x_{N, j}\Vert^2
	= \sum_{i=1}^{N-1}\bar\lambda_i\sum_{j=1}^M \tilde\pi_j \Vert \tilde m^{(N-1)}_j - \tilde x_{N, j}\Vert^2
	\leq \sum_{i=1}^{N-1}\bar\lambda_i \langle c, \pi^i\rangle \\
	=\;& \sum_{i=1}^{N-1}\bar\lambda_i\sum_{j, k} \pi^{i, N}_{j, k}\Vert \tilde m^{(N-1)}_j - x_k^N\Vert^2 
	\leq 2\sum_{i=1}^{N-1}\bar\lambda_i\sum_{j, k} \pi^{i, N}_{j, k}\Big(\Vert \tilde m^{(N-1)}_j -\tilde x_{i, j}\Vert^2 + \Vert \tilde x_{i, j}- x_k^N\Vert^2 \Big)\\
	=\;& 2 \sum_{i=1}^{N-1} \bar\lambda_i \sum_{j=1}^M \tilde\pi_j \Vert \tilde m^{(N-1)}_j - \tilde x_{i, j}\Vert^2 + 2\sum_{i=1}^{N-1}\bar\lambda_i\Wtwo^2(\mu^i, \mu^N).
	\end{align*}
	Inserting this estimate into \eqref{eq:inductive_split_costs},
	we get that
	\begin{align}
	\Phi(\tilde\pi)
	&\leq (1-\lambda_N)(1+2\lambda_N)\sum_{i=1}^{N-1} \bar\lambda_i\sum_{j=1}^M\tilde\pi_j\Vert \tilde m^{(N-1)}_j - \tilde x_{i, j}\Vert^2 + 2(1-\lambda_N)\lambda_N \sum_{i=1}^{N-1} \bar\lambda_i \Wtwo^2(\mu^i, \mu^N) \nonumber\\
	&=(1-\lambda_N)(1+2\lambda_N)\Phi(\tilde\pi^{(N-1)}) + 2\sum_{i=1}^{N-1}\lambda_i\lambda_N \Wtwo^2(\mu^i, \mu^N) \label{eq:greedy_w2_est},
	\end{align}
	where we abuse notation slightly, letting $\Phi(\tilde\pi^{(N-1)})$ be defined by \eqref{eq:mot} with
	\[
	c_\mathrm{MOT}(x_1, \dots, x_{N-1}) = \sum_{s<t}^{N-1}\bar\lambda_s\bar\lambda_t\Vert x_s-x_t\Vert^2.
	\]
	
	After this preparation, set
	\[
	f(N) \coloneqq
	\frac 1 {12}(11N-4-\frac 6 {N-1})
	\]
	to be the right hand side of \eqref{eq:greedy_upper-rand}.
	In the following, we let $\tilde\pi$ be a random variable depending on the random permutation $\sigma$.
	We show inductively that for all $N\geq 2$, it holds
	\[
	\mathbb E[\Phi(\tilde\pi)] \leq f(N)\sum_{s<t}^N\frac 1 {N^2} \Wtwo^2(\mu^s, \mu^t).
	\]
	Since $f(2)=1$, the statement is clear for $N=2$ by definition of Algorithm~\ref{alg:greedy}.
	For $N>2$, note that since the ordering $\sigma$ of the measures is chosen uniformly at random, the probability that $\sigma(N)=i$ is $1/N$ for any $i=1, \dots, N$.
	By the induction hypothesis and \eqref{eq:greedy_w2_est}, we therefore get that
	\begin{align*}
		\mathbb E[\Phi(\tilde\pi)]
		&\leq (1-\frac 1 N)(1+\frac 2 N) \mathbb E[\Phi(\tilde\pi^{(N-1)})]
		+ \sumitoN \frac 1 N\cdot  2\sum_{j\neq i}^N \frac 1 {N^2} \Wtwo^2(\mu^i, \mu^j) \\
		&\leq (1-\frac 1 N)(1+\frac 2 N)\sumitoN\frac 1 N\Big(  f(N-1)\sum_{\substack{s<t \\ s,t\neq i}}^N\frac{1}{(N-1)^2}\Wtwo^2(\mu^s, \mu^t) \Big)
		+ \frac 4 N \sum_{s<t}^N \frac 1 {N^2} \Wtwo^2(\mu^i, \mu^j).
	\end{align*}
	Note that
	\begin{align*}
		&\sumitoN \sum_{\substack{s<t \\ s,t\neq i}}^N \Wtwo^2(\mu^s, \mu^t)
		= \sumitoN \Big( \sum_{s<t}^N \Wtwo^2(\mu^s, \mu^t)- \sum_{j=1}^N \Wtwo^2(\mu^i, \mu^j)\Big) \\
		=\;&N\sum_{s<t}^N \Wtwo^2(\mu^s, \mu^t) - 2\sum_{s<t}^N \Wtwo^2(\mu^s, \mu^t)
		= (N-2)\sum_{s<t}^N \Wtwo^2(\mu^s, \mu^t).
	\end{align*}
	Thus,
	\[
	\mathbb E[\Phi(\tilde\pi)]
	\leq  \Big( (1-\frac 1 N)(1+\frac 2 N)(N-2)\frac 1 N  f(N-1)\Big(\frac{N}{N-1}\Big)^2 + \frac 4 N \Big) \sum_{s<t}^N \frac 1 {N^2}\Wtwo^2(\mu^s, \mu^t),
	\]
	with
	\begin{align*}
		(1-\frac 1 N)(1+\frac 2 N)(N-2) \frac 1 N f(N-1) \Big(\frac{N}{N-1}\Big)^2 + \frac 4 N 
		=\frac{(N-2)(N+2)}{N(N-1)}f(N-1) + \frac 4 N.
	\end{align*}
	By elementary computations, this is equal to $f(N)$, which concludes the induction.
	Finally, using \eqref{eq:2-wass}, we get
	\[
	\frac{\mathbb E[\Phi(\tilde\pi)]}{\Phi(\tilde\pi)}
	\leq \frac{f(N)\sum_{s<t}^N \frac 1 {N^2}\Wtwo^2(\mu^s, \mu^t)}{\sum_{s<t}^N \frac 1 {N^2}\Wtwo^2(\mu^s, \mu^t)}
	= f(N).
	\]
\end{proof}


\subsection{Greedy Algorithm -- Lower Bound}
Finally, we prove the lower bound given by Theorem~\ref{thm:greedy-lower}.
\begin{proof}
	As in the proof of Theorem~\ref{thm:ref-lower}, we give an example on the $1$-dimensional torus that will generalize to two-dimensional Euclidean space using Lemma~\ref{lem:torus}.
	We identify all $x\in \R$ that differ only by a multiple of $2\pi$ and denote by
	\begin{align*}
	s(x) \coloneqq 1_{(0, \pi]} - 1_{(\pi, 2\pi]}
	\end{align*}
	a ``sign function'' on the torus, where $1_A$ denotes the characteristic function on the set $A$.
	Further, we write
	\[
	\vert x\vert \coloneqq
	\begin{cases}
	x,& \quad x\in [0, \pi) \\
	2\pi-x,& \quad x\in [\pi,2\pi).
	\end{cases}
	\]
	Let $M\in \mathbb N$ be large and $\R_{>0}\supset (\varepsilon_i)\to 0$ be a sequence of small enough positive numbers.
	Define inductively for all $i=1, \dots, N$:
	\begin{align*}
	&\begin{cases}
	x_0^i = 0, \quad \tilde m_0^{(i)} = 0 & \text{for } i=1 \\
	x_0^i = \tilde m_0^{(i-1)} + \frac 1 2 \cdot\frac {2\pi}{M} + s(\tilde m_0^{(i-1)})\cdot \varepsilon_i\quad & \text{for } i>1
	\end{cases}
	\\
	&x_j^i = x_0^i + j\cdot \frac {2\pi}{M} \quad \text{for all } j=1, \dots, M-1 \\
	&\mu^i = \sum_{j=1}^M \frac 1 M \delta(x_j^i) \\
	\text{Compute}\quad  &\tilde\pi^{(i)} = \sum_{j=1}^M\frac 1 M \delta(\tilde x_{1, j}, \dots, \tilde x_{i, j}) \quad \text{and} \quad \tilde m_j^{(i)}=\sum_{k=1}^i \frac 1 i \tilde x_{k, j}\quad \text{using Algorithm } \ref{alg:greedy}. 
	\end{align*}
	We use the squared distances on the torus for the cost $c$ and choose $\lambda \equiv 1/N$.
	For an illustration of this example, see Figure~\ref{fig:worst_case_greedy}.
	\begin{figure}[htb]
		\renewcommand\curwidth{.7\textwidth}
		\centering
		\includegraphics[width=\curwidth]{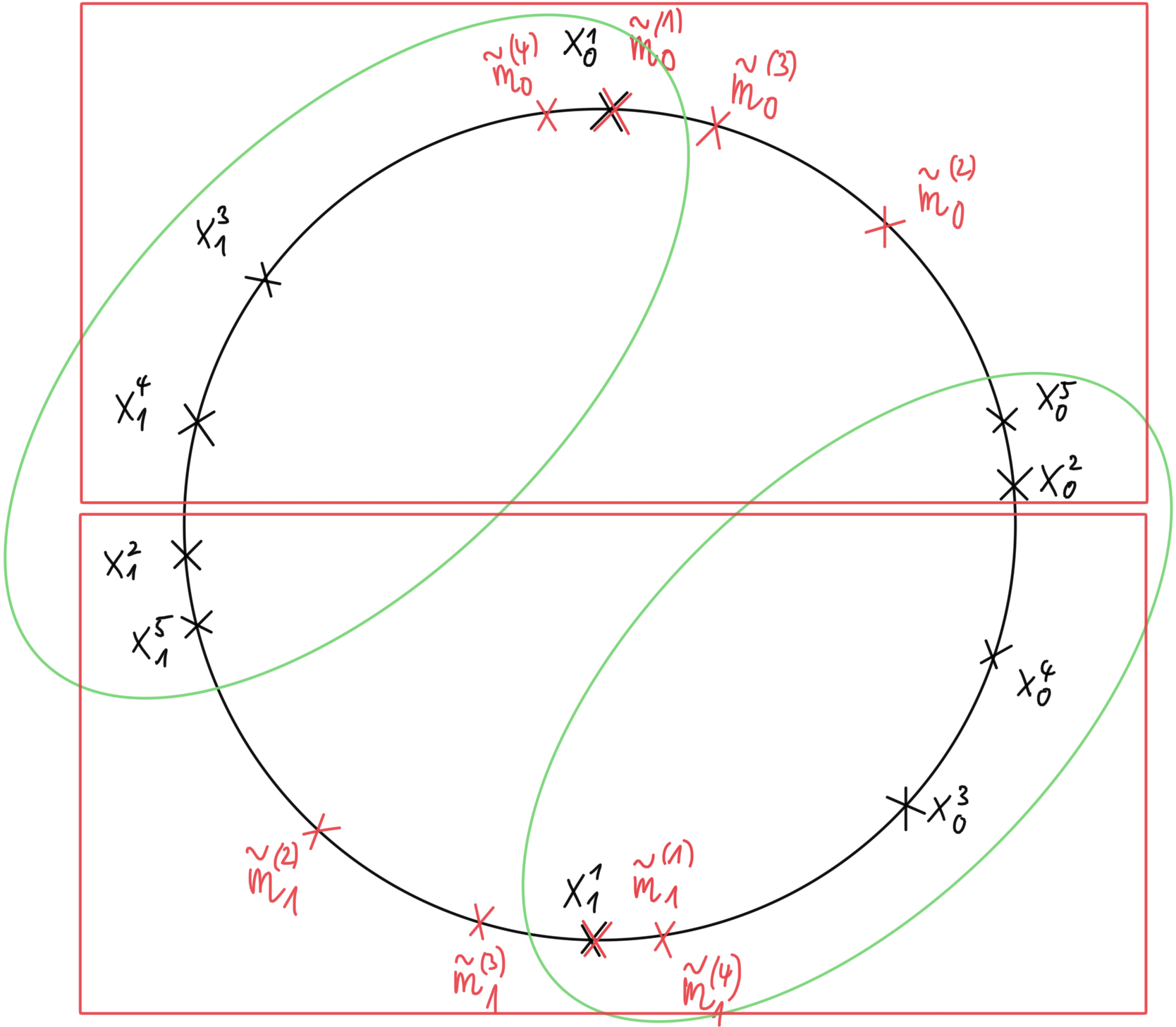}
		\caption{Sketch of the presented example for the greedy heuristic for the case $M=2$. Note that zero is north in the drawing and the positive direction is clock-wise. The points of $N=5$ measures are drawn in black, and the mean-points $\tilde m_j^{(i)}$ (up to the $i=4$) in red. The positions of the points were estimated by eye, precise enough to illustrate the principle. The red rectangles indicate the tuples created by Algorithm~\ref{alg:greedy}, whereas the green ellipses indicate the optimal tuples of $\hat\pi$. Already for such small $N$, we see that the points $x^i_j$ and $\tilde m_j^{(i)}$ are each clustering in a region shifted roughly by $\pi/M$ to each other.}
		\label{fig:worst_case_greedy}
	\end{figure}
	
	Note that for all $i=1, \dots, N$, $x^i_0$ and $x^i_{M-1}$ are the closest points to $\tilde m_0^{(i-1)}$ with almost the same distance to $\tilde m_0^{(i-1)}$.
	However, the last term in the definition of $x^i_0$ is chosen carefully so that whenever $s(\tilde m_0^{(i-1)}) = -1$, the point $x^i_0$, which fulfills $s(x^i_0) = 1$, is slightly closer to $\tilde m_0^{(i-1)}$.
	On the other hand, if $s(\tilde m_0^{(i-1)}) = 1$, the point $x^i_{M-1}$ with $s(x^i_{M-1})=-1$ is slightly closer to $\tilde m_0^{(i-1)}$.
	As a consequence, $\tilde\pi^{(i)} = \sum_{j=1}^M\frac 1 M \delta(\tilde x_{1, j}, \dots, \tilde x_{i, j})$ with $\tilde m_j^{(i)}=\sum_{k=1}^i \frac 1 i \tilde x_{k, j}$ will be so that $\tilde x_{i, 0} = x^i_0$ if $s(\tilde m_0^{(i-1)}) = -1$ and $\tilde x_{i, 0} = x^i_{M-1}$ if $s(\tilde m_0^{(i-1)}) = 1$.
	That is, we always have $s(\tilde x_{i, 0})\neq s(\tilde m_0^{(i-1)})$.
	Intuitively, the sign in front of $\varepsilon_i$ is chosen so that the algorithm takes the ``wrong choice'' between to almost equally good alternatives.
	
	Next, we estimate $\Phi(\hat\pi)$ and $\Phi(\tilde\pi)$ inductively from above and below, respectively.
	Set $\hat m_j = \sumitoN \frac 1 N \hat x_{i, j}$ and $m_j = \sumitoN \frac 1 N x^i_j$.
	Note that for $M$ large enough, $\vert x^i_0\vert$, $\vert x^i_{M-1}\vert$ and thus $\vert\tilde m_0^{(i)}\vert$ are small enough so that they are equal to the standard absolute value function around zero.
	We use the plan
	\[
	\pi=\sum_{j=1}^M \frac 1 M \delta(x_j^1, \dots, x_j^N)
	\]
	with $\Phi(\hat\pi)\leq \Phi(\pi)$ as a lower bound.
	By the symmetry, \eqref{eq:Phi_2} and Lemma~\ref{lem:jensen_eq} it holds
	\begin{align*}
	\Phi(\hat\pi)
	&\leq \Phi(\pi) = \sum_{i=1}^N\sum_{j=1}^M \frac 1 {NM}\vert x_j^i - m_j\vert^2
	= \sumitoN \frac 1 N \vert x^i_0 - m_0\vert^2 \\
	&= \sumitoN \frac 1 N (\vert x^i_0 - \frac 1 2 \frac {2\pi}M \vert^2 - \vert m_0-\frac 1 2 \frac {2\pi}M\vert^2)
	\leq \sumitoN \frac 1 N \vert x^i_0 - \frac \pi M \vert^2 \\
	&= \frac 1 N \Big( \frac\pi M \Big)^2 + \sum_{i=2}^N \frac 1 N \vert x^i_0 - \frac \pi M \vert^2,
	\end{align*}
	where the terms in the sum are equal to
	\begin{align*}
	&\sum_{i=2}^N \frac 1 N \vert \tilde m_0^{(i-1)} + \frac \pi M + s(\tilde m_0^{(i-1)})\cdot \varepsilon_i - \frac \pi M\vert^2
	= \sum_{i=2}^N \frac 1 N \vert \tilde m_0^{(i-1)} + s(\tilde m_0^{(i-1)})\cdot \varepsilon_i \vert^2 \\
	=\;& \sum_{i=2}^N \frac 1 N (\vert \tilde m_0^{(i-1)}\vert + \varepsilon_i )^2
	= \sum_{i=2}^N \frac 1 N \varepsilon_i\Big( 2\vert \tilde m_0^{(i-1)} \vert + \varepsilon_i\Big) + \sum_{i=2}^N \frac 1 N \vert \tilde m_0^{(i-1)} \vert^2.
	\end{align*}
	To derive an upper bound for this, we show by induction that 
	\[
	\vert \tilde m_0^{(i)}\vert \leq \frac 1 i \frac \pi M
	\]
	for all $i=1, \dots, N$.
	The case $i=1$ is clear.
	For $i>1$, it holds by construction that
	\[
	\vert \tilde m_0^{(i-1)} - \tilde x_{i, 0}\vert = \frac \pi M - \varepsilon_i.
	\]
	Since $s(\tilde m_0^{(i-1)}) \neq s(\tilde x_{i, 0})$,
	\[
	\vert \tilde x_{i, 0} \vert
	\leq \vert \tilde m_0^{(i-1)} - \tilde x_{i, 0}\vert
	= \frac \pi M - \varepsilon_i.
	\]
	Furthermore,
	\begin{align*}
	\vert \tilde m_0^{(i)}\vert
	= \Big\vert \sum_{k=1}^i \frac 1 i \tilde x_{k, 0}\Big\vert
	= \Big\vert \frac{i-1}i \tilde m_0^{(i-1)} + \frac 1 i \tilde x_{i, 0}\Big\vert.
	\end{align*}
	Again using $s(m_0^{(i-1)}) \neq s(\tilde x_{i, 0})$, we therefore get
	\[
	\vert \tilde m_0^{(i)}\vert
	\leq \max\Big( \frac{i-1}i \vert \tilde m_0^{(i-1)}\vert, \frac 1 i \vert \tilde x_{i, 0}\vert \Big)
	\leq \max\Big( \frac{i-1}i \vert \tilde m_0^{(i-1)}\vert, \frac 1 i \frac \pi M -\varepsilon_i \Big).
	\]
	We conclude the induction by seeing that using the induction hypothesis,
	\[
	\vert \tilde m_0^{(i)}\vert
	\leq \max\Big( \frac{i-1}i \frac 1 {i-1}\frac\pi M, \frac 1 i \frac \pi M -\varepsilon_i\Big)
	= \frac 1 i \frac \pi M.
	\]
	Thus, for $(\varepsilon_i)$ small enough, $\Phi(\hat\pi)$ is bounded from above by
	\begin{align*}
	\Phi(\hat\pi)
	&\leq \frac 1 N \Big( \frac\pi M \Big)^2
	+ \sum_{i=2}^N \frac 1 N \varepsilon_i\Big( 2\vert \tilde m_0^{(i-1)} \vert + \varepsilon_i\Big) + \sumitoN \frac 1 N \vert \tilde m_0^{(i-1)} \vert^2 \\
	&= \frac 1 N \Big( \frac\pi M \Big)^2
	+ \sum_{i=2}^N \frac 1 N \varepsilon_i\Big( 2\frac 1{i-1}\frac\pi M + \varepsilon_i\Big) + \sum_{i=2}^N \frac 1 N \Big(\frac 1{i-1}\frac\pi M\Big)^2 \\
	&\leq \varepsilon'
	+ \frac 1 N \Big( \frac\pi M \Big)^2
	+ \sumitoN \frac 1 N \Big(\frac 1 i\frac\pi M\Big)^2
	\leq \varepsilon' + \frac 1 N \Big( \frac\pi M \Big)^2\Big( \frac{\pi^2}{6} + 1 \Big)
	\end{align*}
	for any $\varepsilon'>0$.

	
	For $\Phi(\tilde\pi)$, we first note that by construction,
	\[
	\vert x^i_0 - \tilde m_0^{(i-1)}\vert
	= \vert \tilde m_0^{(i-1)} + \frac \pi M + s(\tilde m_0^{(i-1)})\cdot \varepsilon_i - \tilde m_0^{(i-1)}\vert
	\geq \frac \pi M - \varepsilon_i.
	\]
	Further,
	\begin{align*}
	\vert x^i_{M-1} - \tilde m_0^{(i-1)}\vert
	&= \vert x^i_0-\frac {2\pi}M - \tilde m_0^{(i-1)}\vert
	= \vert \tilde m_0^{(i-1)} + \frac \pi M + s(\tilde m_0^{(i-1)})\cdot \varepsilon_i -\frac {2\pi}M- \tilde m_0^{(i-1)}\vert \\
	&= \vert - \frac \pi M + s(\tilde m_{M-1}^{(i-1)})\cdot \varepsilon_i \vert
	\geq \frac \pi M - \varepsilon_i .
	\end{align*}
	Thus,
	\[
	\vert \tilde x_{i, 0} - \tilde m_0^{(i-1)}\vert
	= \min(\vert x^i_0-\tilde m_0^{(i-1)}\vert, \vert x^i_{M-1}-\tilde m_0^{(i-1)}\vert)
	\geq \frac \pi M - \varepsilon_i.
	\]
	Then we can show by induction that for any $2\pi/M>\varepsilon''>0$, if $\varepsilon_i$ is chosen small enough, it holds
	\[
	\Phi(\tilde\pi)
	\geq (1-\frac 1 N H_N)\Big(\frac \pi M - \varepsilon''\Big)^2, \quad \text{where } H_N = \sum_{i=1}^N \frac 1 i.
	\]
	For $N=2$,
	\[
	\Phi(\tilde\pi)
	= \frac 1 4 M\frac 1 M \Big(\frac \pi M - \varepsilon_2\Big)^2
	\geq \Big( 1-\frac 1 2 \Big(1+\frac 1 2\Big) \Big)\Big(\frac \pi M - \varepsilon''\Big)^2
	\]
	for $\varepsilon_2$ small enough.
	For $N>2$, by \eqref{eq:inductive_split_costs},
	\begin{align*}
	\Phi(\tilde\pi)
	&= \sumitoN \sum_{j=1}^M \frac 1 {NM}\vert\tilde x_{i, j}-\tilde m_j\vert^2
	= \sumitoN \frac 1 N \vert \tilde x_{i, 0}-\tilde m_0^{(N)}\vert^2 \\
	&= \Big(1-\frac 1 N\Big)\Big( \frac 1 N \vert \tilde x_{N, 0} - \tilde m_0^{(N-1)}\vert^2 + \sum_{i=1}^{N-1} \frac 1 {N-1}\vert\tilde x_{i, 0}-\tilde m_0^{(N-1)} \vert^2\Big).
	\end{align*}
	By the induction hypothesis and the computations above, we get that this is greater or equal to
	\[
	\Big(1-\frac 1 N\Big)\Big( \frac 1 N \Big( \frac \pi M - \varepsilon_N \Big)^2 +  (1-\frac 1 {N-1}H_{N-1}) \Big( \frac \pi M - \varepsilon'' \Big)^2 \Big).
	\]
	By choosing $\varepsilon_N\leq \varepsilon''$, this is
	\[
	\geq \Big( \frac \pi M - \varepsilon'' \Big)^2 \Big( 1- \frac 1 N \Big) \Big( 1 -\frac 1 {N-1} H_{N-1}+ \frac 1 N \Big),
	\]
	which, by elementary computations, equals
	\[
	\Big( \frac \pi M - \varepsilon'' \Big)^2\Big( 1-\frac 1 N H_N \Big).
	\]
	
	
	The relative error can now be bounded from below by
	\begin{align*}
	\frac{\Phi(\tilde\pi)}{\Phi(\hat\pi)}
	\geq \frac{( \frac \pi M - \varepsilon'' )^2( 1-\frac 1 N H_N )}{\varepsilon' + \frac 1 N ( \frac\pi M )^2( \frac{\pi^2}{6} + 1)}
	\geq \frac{N- H_N}{( \frac M \pi)^2 N\varepsilon' +\frac{\pi^2}{6} + 1} -  \frac{\varepsilon''\frac{2\pi}{M}(1-\frac 1 N H_N)}{\frac 1 N ( \frac\pi M )^2( \frac{\pi^2}{6} + 1)}.
	\end{align*}
	If $(\varepsilon_i)$ is chosen small enough, then this is at least
	\begin{align*}
	\frac{N-H_N}{\frac{\pi^2}{6} + 1} - \varepsilon
	\end{align*}
	for any $\varepsilon>0$.
	Using
	\[
	N-H_N \geq N - (1+\frac 1 2 + \frac 1 3 (N-2))
	\]
	for $N\geq 2$ and
	\[
	9 = 3^2 < \pi^2 < \Big(\frac {22} 7 \Big)^2 = \frac{484}{49} < 10,
	\]
	for $\varepsilon$ small enough, by elementary computations, this is at least
	\[
	\frac 1 4 N - \frac 1 3.
	\]
	Finally, we use Lemma~\ref{lem:torus} as in the proof of Theorem~\ref{thm:ref-lower} to conclude the proof.
\end{proof}
	
\end{appendix}

\end{document}

%% file: preamble.tex
\usepackage[margin=2.5cm]{geometry}
\usepackage[parfill]{parskip}
\usepackage[utf8]{inputenc}
\usepackage{amsmath,amssymb,amsfonts,amsthm,mathtools}
\usepackage{dsfont}
\usepackage{graphicx}
\usepackage[caption=false]{subfig}
\usepackage{caption}
\usepackage{algorithm, algpseudocode}
\usepackage{color}
\usepackage{booktabs}
\usepackage{pifont}
\usepackage{url}
\usepackage{colonequals}

\newtheorem{theorem}{Theorem}[section]
\newtheorem{lemma}[theorem]{Lemma}

\newtheorem{definition}[theorem]{Definition}
\newtheorem{example}[theorem]{Example}

\newtheorem{proposition}[theorem]{Proposition}

\makeatletter
\def\thm@space@setup{%
	\thm@preskip=\parskip \thm@postskip=0pt
}
\makeatother

\normalfont

\newcommand{\R}{\ensuremath{\mathbb{R}}}

\newcommand{\supp}{\textnormal{supp}}

\newcommand{\push}{(M_\lambda)_\#}
\newcommand{\cmark}{\ding{51}}
\newcommand{\xmark}{\ding{55}}

\DeclareMathOperator*{\argmax}{arg\,max}

\DeclareMathOperator*{\argmin}{argmin}

\newcommand{\Wtwo}{\mathcal{W}_2}
\newcommand{\sumitoN}{\sum_{i=1}^N}
\newcommand{\sumitoNlambdai}{\sumitoN \lambda_i}

\newcommand\curwidth{0.38\textwidth}

\newcommand\curfolder{img}

\numberwithin{equation}{section} 
\AtBeginDocument{
	\label{CorrectFirstPageLabel}
	
}

%% file: tables/results.tex
\begin{tabular}{lrrrrrc}
\toprule
{} & $\Psi(\tilde\nu)$ & $\Psi(\tilde\nu)/\Psi(\hat\nu)$ &  error &  runtime & error$\cdot$runtime & free support \\
\midrule
Exact       &           0.02666 &                          1.0000 & 0.0000 & 18187.67 &              0.0000 &       \cmark \\
Reference   &           0.02680 &                          1.0050 & 0.0050 &     \textbf{0.05} &              \textbf{0.0003} &       \cmark \\
Greedy      &           0.02669 &                          1.0012 & \textbf{0.0012} &     0.34 &              0.0004 &       \cmark \\
IBP         &           0.02723 &                          1.0214 & 0.0214 &     0.07 &              0.0016 &       \xmark \\
Debiased    &           0.02675 &                          1.0033 & 0.0033 &     1.19 &              0.0039 &       \xmark \\
Product     &           0.02688 &                          1.0082 & 0.0082 &    17.58 &              0.1440 &       \xmark \\
MAAIPM      &           0.02672 &                          1.0020 & 0.0020 &   158.51 &              0.3091 &       \xmark \\
FastIBP     &           0.02753 &                          1.0323 & 0.0323 &   111.03 &              3.5899 &       \xmark \\
Frank--Wolfe &           0.02870 &                          1.0763 & 0.0763 &    56.59 &              4.3168 &       \cmark \\
\bottomrule
\end{tabular}

%% file: mot_algorithms.bbl
\begin{thebibliography}{10}

\bibitem{AC11barycenters}
M.~Agueh and G.~Carlier.
\newblock Barycenters in the {W}asserstein space.
\newblock {\em SIAM J. Math. Anal.}, 43(2):904--924, 2011.

\bibitem{AMO93networkflows}
R.~K. Ahuja, T.~L. Magnanti, and J.~B. Orlin.
\newblock {\em Network Flows: Theory, Algorithms, and Applications}.
\newblock Prentice-Hall, Inc., USA, 1993.

\bibitem{AB21fixedd}
J.~M. Altschuler and E.~Boix-Adsera.
\newblock {W}asserstein barycenters can be computed in polynomial time in fixed
  dimension.
\newblock {\em J. Mach. Learn. Res.}, 22(44):1--19, 2021.

\bibitem{AB21nphard}
J.~M. Altschuler and E.~Boix-Adsera.
\newblock {W}asserstein barycenters are {NP}-hard to compute.
\newblock {\em SIAM J. Math. Data Sci.}, to appear.

\bibitem{ABM16brute}
E.~Anderes, S.~Borgwardt, and J.~Miller.
\newblock Discrete {W}asserstein barycenters: optimal transport for discrete
  data.
\newblock {\em Math. Methods Oper. Res.}, 84(2):389--409, 2016.

\bibitem{GWLOT21BBS}
F.~Beier, R.~Beinert, and G.~Steidl.
\newblock On a linear {G}romov--{W}asserstein distance.
\newblock {\em arXiv preprint arXiv:2112.11964}, 2021 (under review).

\bibitem{UMOT21BLNS}
F.~Beier, J.~von Lindheim, S.~Neumayer, and G.~Steidl.
\newblock Unbalanced multi-marginal optimal transport.
\newblock {\em arXiv preprint arXiv:2103.10854}, 2021 (under review).

\bibitem{BCC15IBP}
J.-D. Benamou, G.~Carlier, M.~Cuturi, L.~Nenna, and G.~Peyr\'{e}.
\newblock Iterative {B}regman projections for regularized transportation
  problems.
\newblock {\em SIAM J. Sci. Comput.}, 37(2):A1111--A1138, 2015.

\bibitem{flows19BCN}
J.-D. Benamou, G.~Carlier, and L.~Nenna.
\newblock Generalized incompressible flows, multi-marginal transport and
  {S}inkhorn algorithm.
\newblock {\em Numer. Math.}, 142(1):33--54, 2019.

\bibitem{BPPH11networksimplex}
N.~Bonneel, M.~van~de Panne, S.~Paris, and W.~Heidrich.
\newblock Displacement interpolation using {L}agrangian mass transport.
\newblock In {\em Proceedings of the 2011 SIGGRAPH Asia Conference}, SA '11,
  New York, NY, USA, 2011. Association for Computing Machinery.

\bibitem{B20LPapprox}
S.~Borgwardt.
\newblock An lp-based, strongly-polynomial 2-approximation algorithm for sparse
  {W}asserstein barycenters.
\newblock {\em Oper. Res.}, pages 1--41, 2020.

\bibitem{improvedLP20BP}
S.~Borgwardt and S.~Patterson.
\newblock Improved linear programs for discrete barycenters.
\newblock {\em INFORMS J. Optim.}, 2(1):14--33, 2020.

\bibitem{columngen22BP}
S.~Borgwardt and S.~Patterson.
\newblock A column generation approach to the discrete barycenter problem.
\newblock {\em Discrete Optim.}, 43:Paper No. 100674, 16, 2022.

\bibitem{DFT12B}
G.~Buttazzo, L.~De~Pascale, and P.~Gori-Giorgi.
\newblock Optimal-transport formulation of electronic density-functional
  theory.
\newblock {\em Phys. Rev. A}, 85(6):062502, 2012.

\bibitem{HK-LOT22CCST}
T.~Cai, J.~Cheng, B.~Schmitzer, and M.~Thorpe.
\newblock The linearized {H}ellinger--{K}antorovich distance.
\newblock {\em SIAM J. Imaging Sci.}, 15(1):45--83, 2022.

\bibitem{CE10teams}
G.~Carlier and I.~Ekeland.
\newblock Matching for teams.
\newblock {\em Econom. Theory}, 42(2):397--418, 2010.

\bibitem{CMKMV14textures}
M.~Cimpoi, S.~Maji, I.~Kokkinos, S.~Mohamed, , and A.~Vedaldi.
\newblock Describing textures in the wild.
\newblock In {\em Proceedings of the {IEEE} Conf. on Computer Vision and
  Pattern Recognition ({CVPR})}, 2014.

\bibitem{DFT13C}
C.~Cotar, G.~Friesecke, and C.~Kl\"{u}ppelberg.
\newblock Density functional theory and optimal transportation with {C}oulomb
  cost.
\newblock {\em Comm. Pure Appl. Math.}, 66(4):548--599, 2013.

\bibitem{CD14fast}
M.~Cuturi and A.~Doucet.
\newblock Fast computation of {W}asserstein barycenters.
\newblock In {\em International conference on machine learning}, pages
  685--693. PMLR, 2014.

\bibitem{DT21complexitybounds}
D.~Dvinskikh and D.~Tiapkin.
\newblock Improved complexity bounds in {W}asserstein barycenter problem.
\newblock In {\em International Conference on Artificial Intelligence and
  Statistics}, pages 1738--1746. PMLR, 2021.

\bibitem{sensors20H}
F.~Elvander, I.~Haasler, A.~Jakobsson, and J.~Karlsson.
\newblock Multi-marginal optimal transport using partial information with
  applications in robust localization and sensor fusion.
\newblock {\em Signal Process.}, 171:107474, 2020.

\bibitem{flamary2021pot}
R.~Flamary, N.~Courty, A.~Gramfort, M.~Z. Alaya, A.~Boisbunon, S.~Chambon,
  L.~Chapel, A.~Corenflos, K.~Fatras, N.~Fournier, L.~Gautheron, N.~T. Gayraud,
  H.~Janati, A.~Rakotomamonjy, I.~Redko, A.~Rolet, A.~Schutz, V.~Seguy, D.~J.
  Sutherland, R.~Tavenard, A.~Tong, and T.~Vayer.
\newblock Pot: Python optimal transport.
\newblock {\em J. Mach. Learn. Res.}, 22(78):1--8, 2021.

\bibitem{GS98}
W.~Gangbo and A.~\'{S}wi\c{e}ch.
\newblock Optimal maps for the multidimensional {M}onge--{K}antorovich problem.
\newblock {\em Comm. Pure Appl. Math.}, 51(1):23--45, 1998.

\bibitem{GWXY19MAAIPM}
D.~Ge, H.~Wang, Z.~Xiong, and Y.~Ye.
\newblock Interior-point methods strike back: {S}olving the {W}asserstein
  barycenter problem.
\newblock In H.~Wallach, H.~Larochelle, A.~Beygelzimer, F.~d'~Alch\'{e}-Buc,
  E.~Fox, and R.~Garnett, editors, {\em Advances in Neural Information
  Processing Systems}, volume~32. Curran Associates, Inc., 2019.

\bibitem{tree21HRCK}
I.~Haasler, A.~Ringh, Y.~Chen, and J.~Karlsson.
\newblock Multimarginal optimal transport with a tree-structured cost and the
  {S}chr\"{o}dinger bridge problem.
\newblock {\em SIAM J. Control Optim.}, 59(4):2428--2453, 2021.

\bibitem{HMZ20randomized}
F.~Heinemann, A.~Munk, and Y.~Zemel.
\newblock Randomised {W}asserstein barycenter computation: Resampling with
  statistical guarantees.
\newblock {\em SIAM J. Math. Data Sci.}, to appear.

\bibitem{HLPR21gotex}
A.~Houdard, A.~Leclaire, N.~Papadakis, and J.~Rabin.
\newblock A generative model for texture synthesis based on optimal transport
  between feature distributions.
\newblock {\em arXiv preprint arXiv:2007.03408}, 2021.

\bibitem{ISS21dimreduction}
Z.~Izzo, S.~Silwal, and S.~Zhou.
\newblock Dimensionality reduction for {W}asserstein barycenter.
\newblock In A.~Beygelzimer, Y.~Dauphin, P.~Liang, and J.~W. Vaughan, editors,
  {\em Advances in Neural Information Processing Systems}, 2021.

\bibitem{JCG20debiased}
H.~Janati, M.~Cuturi, and A.~Gramfort.
\newblock Debiased {S}inkhorn barycenters.
\newblock In H.~D. III and A.~Singh, editors, {\em Proceedings of the 37th
  International Conference on Machine Learning}, volume 119 of {\em Proceedings
  of Machine Learning Research}, pages 4692--4701. PMLR, 13--18 Jul 2020.

\bibitem{complexitybarycenters19KTDDGU}
A.~Kroshnin, N.~Tupitsa, D.~Dvinskikh, P.~Dvurechensky, A.~Gasnikov, and
  C.~Uribe.
\newblock On the complexity of approximating {W}asserstein barycenters.
\newblock In K.~Chaudhuri and R.~Salakhutdinov, editors, {\em Proceedings of
  the 36th International Conference on Machine Learning}, volume~97 of {\em
  Proceedings of Machine Learning Research}, pages 3530--3540. PMLR, 09--15 Jun
  2019.

\bibitem{cont20LGYS}
L.~Li, A.~Genevay, M.~Yurochkin, and J.~M. Solomon.
\newblock Continuous regularized {W}asserstein barycenters.
\newblock In H.~Larochelle, M.~Ranzato, R.~Hadsell, M.~F. Balcan, and H.~Lin,
  editors, {\em Advances in Neural Information Processing Systems}, volume~33,
  pages 17755--17765. Curran Associates, Inc., 2020.

\bibitem{LHXCJ20fastIBP}
T.~Lin, N.~Ho, X.~Chen, M.~Cuturi, and M.~Jordan.
\newblock Fixed-support {W}asserstein barycenters: Computational hardness and
  fast algorithm.
\newblock In H.~Larochelle, M.~Ranzato, R.~Hadsell, M.~F. Balcan, and H.~Lin,
  editors, {\em Advances in Neural Information Processing Systems}, volume~33,
  pages 5368--5380. Curran Associates, Inc., 2020.

\bibitem{lin2019complexity}
T.~Lin, N.~Ho, M.~Cuturi, and M.~I. Jordan.
\newblock On the complexity of approximating multimarginal optimal transport.
\newblock {\em arXiv preprint arXiv:1910.00152}, 2019.

\bibitem{LSPC2019frankwolfe}
G.~Luise, S.~Salzo, M.~Pontil, and C.~Ciliberto.
\newblock Sinkhorn barycenters with free support via {F}rank--{W}olfe
  algorithm.
\newblock In H.~Wallach, H.~Larochelle, A.~Beygelzimer, F.~d'~Alch\'{e}-Buc,
  E.~Fox, and R.~Garnett, editors, {\em Advances in Neural Information
  Processing Systems}, volume~32, pages 9322--9333. Curran Associates, Inc.,
  2019.

\bibitem{LOT20MDC}
Q.~M\'erigot, A.~Delalande, and F.~Chazal.
\newblock Quantitative stability of optimal transport maps and linearization of
  the 2-{W}asserstein space.
\newblock In S.~Chiappa and R.~Calandra, editors, {\em Proceedings of the
  Twenty Third International Conference on Artificial Intelligence and
  Statistics}, volume 108 of {\em Proceedings of Machine Learning Research},
  pages 3186--3196. PMLR, 26--28 Aug 2020.

\bibitem{LOT21MC}
C.~Moosm{\"u}ller and A.~Cloninger.
\newblock Linear optimal transport embedding: Provable fast {W}asserstein
  distance computation and classification for nonlinear problems.
\newblock {\em arXiv preprint arXiv:2008.09165}, 2020.

\bibitem{stataspects19PZ}
V.~M. Panaretos and Y.~Zemel.
\newblock Statistical aspects of {W}asserstein distances.
\newblock {\em Annu. Rev. Stat. Appl.}, 6:405--431, 2019.

\bibitem{mot15P}
B.~Pass.
\newblock Multi-marginal optimal transport: theory and applications.
\newblock {\em ESAIM Math. Model. Numer. Anal.}, 49(6):1771--1790, 2015.

\bibitem{PC19book}
G.~Peyr{\'e}, M.~Cuturi, et~al.
\newblock Computational optimal transport: With applications to data science.
\newblock {\em Found. Trends Mach. Learn.}, 11(5-6):355--607, 2019.

\bibitem{PRV20oncomputation}
G.~Puccetti, L.~R\"{u}schendorf, and S.~Vanduffel.
\newblock On the computation of {W}asserstein barycenters.
\newblock {\em J. Multivariate Anal.}, 176:104581, 16, 2020.

\bibitem{iPAM21QP}
Y.~Qian and S.~Pan.
\newblock An inexact {PAM} method for computing {W}asserstein barycenter with
  unknown supports.
\newblock {\em Comput. Appl. Math.}, 40(2):Paper No. 45, 29, 2021.

\bibitem{texturemix11}
J.~Rabin, G.~Peyr{\'e}, J.~Delon, and M.~Bernot.
\newblock Wasserstein barycenter and its application to texture mixing.
\newblock In {\em International Conference on Scale Space and Variational
  Methods in Computer Vision}, pages 435--446. Springer, 2011.

\bibitem{S15otapplied}
F.~Santambrogio.
\newblock {\em Optimal transport for applied mathematicians}, volume~87 of {\em
  Progress in Nonlinear Differential Equations and their Applications}.
\newblock Birkh\"{a}user/Springer, Cham, 2015.
\newblock Calculus of variations, PDEs, and modeling.

\bibitem{convolutional15SPCetal}
J.~Solomon, F.~de~Goes, G.~Peyr\'{e}, M.~Cuturi, A.~Butscher, A.~Nguyen, T.~Du,
  and L.~Guibas.
\newblock Convolutional {W}asserstein distances: Efficient optimal
  transportation on geometric domains.
\newblock {\em ACM Trans. Graph.}, 34(4), jul 2015.

\bibitem{bayes18}
S.~Srivastava, C.~Li, and D.~B. Dunson.
\newblock Scalable {B}ayes via barycenter in {W}asserstein space.
\newblock {\em J. Mach. Learn. Res.}, 19:Paper No. 8, 35, 2018.

\bibitem{TSKRY21treesliced}
Y.~Takezawa, R.~Sato, Z.~Kozareva, S.~Ravi, and M.~Yamada.
\newblock Fixed support tree-sliced {W}asserstein barycenter.
\newblock {\em arXiv preprint arXiv:2109.03431}, 2021.

\bibitem{frechettemplates05T}
A.~Trouv\'{e} and L.~Younes.
\newblock Local geometry of deformable templates.
\newblock {\em SIAM J. Math. Anal.}, 37(1):17--59, 2005.

\bibitem{frechetpersistence14}
K.~Turner, Y.~Mileyko, S.~Mukherjee, and J.~Harer.
\newblock Fr\'{e}chet means for distributions of persistence diagrams.
\newblock {\em Discrete Comput. Geom.}, 52(1):44--70, 2014.

\bibitem{LOT13WSBOR}
W.~Wang, D.~Slep{\v{c}}ev, S.~Basu, J.~A. Ozolek, and G.~K. Rohde.
\newblock A linear optimal transportation framework for quantifying and
  visualizing variations in sets of images.
\newblock {\em Int. J. Comput. Vis.}, 101(2):254--269, 2013.

\bibitem{LJDK21linear}
L.~Yang, J.~Li, D.~Sun, and K.-C. Toh.
\newblock A fast globally linearly convergent algorithm for the computation of
  {W}asserstein barycenters.
\newblock {\em J. Mach. Learn. Res.}, 22(21):1--37, 2021.

\bibitem{YWWL17badmm}
J.~Ye, P.~Wu, J.~Z. Wang, and J.~Li.
\newblock Fast discrete distribution clustering using {W}asserstein barycenter
  with sparse support.
\newblock {\em IEEE Trans. Signal Process.}, 65(9):2317--2332, 2017.

\bibitem{frechet19procrustes}
Y.~Zemel and V.~M. Panaretos.
\newblock Fr\'{e}chet means and {P}rocrustes analysis in {W}asserstein space.
\newblock {\em Bernoulli}, 25(2):932--976, 2019.

\end{thebibliography}
